\documentclass[12pt]{amsart}
\usepackage{amscd,verbatim}
\usepackage{amssymb,amsmath}
\usepackage[all]{xy}
\usepackage{url}
\usepackage{enumerate}
\usepackage{bm}
\usepackage{mathrsfs}
\usepackage{color}

\newcommand{\ul}{\underline}
\newcommand{\ol}{\overline}
\newcommand{\wt}{\widetilde}
\newcommand{\pd}{{\operatorname{pd}}}
\newcommand{\ku}{{\operatorname{k\ddot{u}}}}
\newcommand{\co}{{\operatorname{cr}}}

\newcommand{\A}{\mathbb{A}}

\newcommand{\C}{\mathbb{C}}

\newcommand{\G}{\mathbb{G}}
\renewcommand{\L}{\mathbb{L}}

\renewcommand{\P}{\mathbb{P}}
\newcommand{\Q}{\mathbb{Q}}

\newcommand{\Z}{\mathbb{Z}}

\newcommand{\sC}{\mathscr{C}}

\newcommand{\sH}{\mathscr{H}}

\newcommand{\sM}{\mathscr{M}}
\newcommand{\sO}{\mathscr{O}}

\newcommand{\bZ}{\mathbb{Z}}

\newcommand{\op}{{\operatorname{op}}}
\newcommand{\Ab}{\operatorname{\mathbf{Ab}}}
\newcommand{\Mod}{\operatorname{\mathbf{Mod}}}
\newcommand{\Hom}{\operatorname{Hom}}
\newcommand{\Ext}{\operatorname{Ext}}
\newcommand{\Tor}{{\operatorname{Tor}}}
\newcommand{\fr}{{\operatorname{fr}}}

\newcommand{\Coker}{\operatorname{Coker}}
\newcommand{\coker}{\operatorname{Coker}}
\renewcommand{\Im}{\operatorname{Im}}

\newcommand{\ch}{\operatorname{ch}}

\newcommand{\tr}{{\operatorname{tr}}}

\newcommand{\id}{{\operatorname{id}}}
\newcommand{\pr}{\operatorname{pr}}
\renewcommand{\lim}{\operatornamewithlimits{\varprojlim}}
\newcommand{\colim}{\operatornamewithlimits{\varinjlim}}

\newcommand{\rank}{\operatorname{rank}}

\newcommand{\Pic}{\operatorname{Pic}}
\newcommand{\Alb}{\operatorname{Alb}}
\newcommand{\NS}{\operatorname{NS}}
\newcommand{\Br}{\operatorname{Br}}
\newcommand{\Spec}{\operatorname{Spec}}

\newcommand{\Fld}{\operatorname{\mathbf{Fld}}}
\newcommand{\Sch}{\operatorname{\mathbf{Sch}}}
\newcommand{\Sm}{\operatorname{\mathbf{Sm}}}

\newcommand{\SmProj}{\operatorname{\mathbf{SmProj}}}

\newcommand{\Zar}{{\operatorname{Zar}}}
\newcommand{\Nis}{{\operatorname{Nis}}}
\newcommand{\et}{{\operatorname{\acute{e}t}}}
\newcommand{\cyc}{{\operatorname{cyc}}}
\newcommand{\Gal}{{\operatorname{Gal}}}

\newcommand{\CH}{{\operatorname{CH}}}
\newcommand{\Chow}{\operatorname{\mathbf{Chow}}}
\newcommand{\GL}{{\operatorname{GL}}}

\newcommand{\Cor}{\operatorname{\mathbf{Cor}}}

\newcommand{\PST}{{\operatorname{\mathbf{PST}}}}

\newcommand{\fg}{{\operatorname{fg}}}
\newcommand{\alg}{{\operatorname{ac}}}
\newcommand{\ur}{{\operatorname{ur}}}
\newcommand{\eff}{{\operatorname{eff}}}
\newcommand{\bir}{{\operatorname{bir}}}
\newcommand{\nor}{{\operatorname{nor}}}
\newcommand{\lMod}{\Mod_\Lambda}

\theoremstyle{plain}
\newtheorem{theorem}{Theorem}[section]
\newtheorem{proposition}[theorem]{Proposition}
\newtheorem{lemma}[theorem]{Lemma}
\newtheorem{corollary}[theorem]{Corollary}

\newtheorem{problem}[theorem]{Problem}

\theoremstyle{definition}
\newtheorem{definition}[theorem]{Definition}
\newtheorem{example}[theorem]{Example}
\newtheorem{remark}[theorem]{Remark}

\newtheorem{setting}[theorem]{Setting}

\numberwithin{equation}{section}
\begin{document}

\title[Torsion birational motives of surfaces]
{Torsion birational motives of surfaces 
\\ and unramified cohomology}

\dedicatory{In memory of Noriyuki Suwa} 

\author[K. Sato]{Kanetomo Sato}
\address{Department of Mathematics,
Chuo University, 1-13-27 Kasuga,
Bunkyo-ku, Tokyo 112-8551, Japan}
\email{kanetomo@math.chuo-u.ac.jp}

\author[T. Yamazaki]{Takao Yamazaki}
\address{Department of Mathematics,
Chuo University, 1-13-27 Kasuga,
Bunkyo-ku, Tokyo 112-8551, Japan}
\email{ytakao@math.chuo-u.ac.jp}

\date{\today}

 
\thanks{
The first author is supported by JSPS KAKENHI Grant (JP20K03566). 
The second author is supported by JSPS KAKENHI Grant (JP21K03153). 
}


\begin{abstract}
Let $S$ and $T$ be smooth projective varieties over
an algebraically closed field $k$.
Suppose that $S$ is a surface 
admitting a decomposition of the diagonal.
We show that,
away from the characteristic of $k$,
if an algebraic correspondence $T \to S$
acts trivially on the unramified cohomology,
then it acts trivially on 
any normalized, birational, and motivic functor.
This generalizes Kahn's result on the torsion order of $S$.
We also exhibit an example of $S$ over $\C$
for which $S \times S$ violates the integral Hodge conjecture.
\end{abstract}

\maketitle

\section{Introduction}
Let $k$ be an algebraically closed field, and
let $\Chow^{\eff}_{\Z}$ be the covariant category of effective Chow motives over $k$ with $\Z$-coefficients.
Until \S \ref{sect:intro-s3}
we assume 
the characteristic $p$ of $k$ is zero for simplicity,
although most results remain valid 
away from $p$ if $p>0$.

\subsection{Main exact sequence}
Recall that a smooth projective variety $X$ over $k$
is said to admit a \emph{decomposition of the diagonal}
if the degree map induces an isomorphism
$\CH_0(X_{k(X)}) \otimes \Q \cong \Q$,
where $k(X)$ denotes the total ring of fractions of $X$.
This condition implies that $X$ is connected, and
$H^0(X, \Omega_{X/k}^1)=H^0(X, \Omega_{X/k}^2)=0$.
If $\dim X=2$, Bloch's conjecture predicts the converse (see \S \ref{sect:dec-diag} for details).

\par
Let $S$ be a projective smooth surface over $k$ which admits a decomposition of the diagonal.
In his paper \cite{K1}, Kahn introduced a new category $\Chow^{\nor}_{\Z}$,
the category of normalized birational motives,
which is defined as a quotient category of $\Chow^{\eff}_{\Z}$
and has the property that there is a canonical isomorphism
\[ \Chow^{\nor}_{\Z}(T,S)
 \cong \CH_0(S_{k(T)})_\Tor \qquad \text{(cf.\ \eqref{eq:CH0S-CHtor})} \]
for any smooth projective variety $T$ over $k$.
By this isomorphism for $T=S$, the motive of $S$ is a torsion object in $\Chow^{\nor}_{\Z}$
(cf. Definition \ref{def:tor-ord}).
To compute its order, he established an exact sequence
\begin{equation}\label{eq1-1}
 0 \to \Chow^{\nor}_{\Z}(S,S) \to 
\Tor(H_\ur^1(S), H_\ur^2(S))^{\oplus 2}
 \to
H^3_\ur(S \times S) \to 0
\end{equation}
in \cite[Corollary 6.4(a)]{K1}, cf.\ Example \ref{ex:Bruno} below.
Here for a smooth scheme $X$ over $k$ and $i \in \Z_{> 0}$,
 $H^i_\ur(X)$ is the \emph{unramified cohomology} of $X$, defined as follows:
\begin{equation}\label{eq:unram-coh-intro}
H_\ur^i(X):=H^0_\Zar(X, \sH^i),
\end{equation}
where 
$\sH^i$ is the Zariski sheaf on $X$ associated to
the presheaf $U \mapsto H^i_\et(U, \Q/\Z(i-1))$.
As is well-known, we have
$H^1_\ur(X) \cong H^1_\et(X, \Q/\Z)$
and 
$H^2_\ur(X) \cong \Br(X)$, 
the Brauer group of $X$
(see \S \ref{sect:unram-coh} for details). %

Kahn deduced \eqref{eq1-1} by applying $T=S$
to a complicated result \cite[Theorem 6.3]{K1}
that involves $\Chow^{\nor}_{\Z}(T,S)$ 
for a general smooth projective variety $T$ over $k$.
Attempting to foster a better understanding of it,
we found the following simple statement.
(See Remark \ref{ex:Bruno2} below for more discussion.)

\begin{theorem}[Theorem \ref{thm:main2-full}]
\label{thm:main2}
Let $k$ and $S$ be as above, and let $T$ be a smooth projective variety over $k$.
Then there is an exact sequence
\begin{equation}\label{eq:Vishik-ex-seq}
0 \to 
\Chow^{\nor}_{\Z}(T,S)
 \to 
\bigoplus_{i=1, 2} \ \Hom(H_\ur^i(S), H_\ur^i(T)) 
\to H^3_\ur(S \times T) \to 0.
\end{equation}
\end{theorem}

We shall prove the exactness of \eqref{eq:Vishik-ex-seq} by computing the image of the cycle class map
\[ \CH_0(S_{k(T)})_\Tor \longrightarrow H^4_\et(S_{k(T)},\mu_{m}^{\otimes 2}) \]
for a sufficiently large $m$, using Vishik's method \cite{Vishik},
which gives an alternative proof of \eqref{eq1-1}.
%

\subsection{Motivic, birational, and normalized functors}
\label{sect:intro-s2}
Recall from \cite{K1}
that a contravariant functor $F$ defined on 
the category of smooth projective varieties over $k$
and with values in the category of abelian groups
is called 
\begin{itemize}
\item \emph{motivic} 
if $F$ factors through an additive functor on $\Chow^\eff_\Z$,
\item \emph{birational} 
if $F(f)$ is an isomorphism 
for any birational morphism $f$,
and
\item \emph{normalized} 
if $F(\Spec k)=0$.
\end{itemize}
A normalized, birational, and motivic functor is equivalent to a functor which factors through an additive functor on $\Chow^\nor_\Z$. See \S \ref{sect:mot-inv} for details.
Fundamental examples of such functors
include $H^0(-, \Omega^i_{-/k})$ for $i>0$ and
the unramified cohomology \eqref{eq:unram-coh-intro}.
We deduce the following result from
the injectivity of the first map in \eqref{eq:Vishik-ex-seq}:

\begin{theorem}[Theorem \ref{thm:main1-full}]
\label{thm:main1}
Let $S$ and $T$ be smooth projective varieties over $k$.
Suppose that $S$ admits a decomposition of the diagonal
and $\dim S=2$.
Let $f : T \to S$
be an algebraic correspondence
such that 
$H^i_\ur(f) : H^i_\ur(S) \to H^i_\ur(T)$ 
vanishes for $i=1, 2$.
Then 
$F(f) : F(S) \to F(T)$ vanishes
for any normalized, birational, and motivic functor $F$.
\end{theorem}


Theorem \ref{thm:main1} will be applied 
to the K3 cover $f : T \to S$
of an Enriques surface $S$ over $\C$
to interpret Beauville's result \cite{B}
in Example \ref{ex:Beauville} below.

\subsection{Explicit computation of 
$\boldsymbol{\CH_0(S_{k(S)})_{\Tor}}$ and $\boldsymbol{H^3_\ur(S \times S)}$}
\label{sect:intro-s3}
The groups appearing in \eqref{eq:Vishik-ex-seq}
attracted some attention.
Kahn \cite[p. 840, footnote]{K1} raised a question
asking the structure of 
$\CH_0(S_{k(S)})_{\Tor}$ for an Enriques surface $S$.
The group $H^3_\ur(X)$ for a smooth projective variety $X$
over $\C$ is studied by many authors,
since it gives an obstruction to the integral Hodge conjecture
by a theorem of
Colliot-Th\'el\`ene and Voisin \cite{CTV}
(see Theorem \ref{thm:CTV}).
Therefore there is some interest 
in making each term in \eqref{eq:Vishik-ex-seq} explicit.
In this direction, we obtain the following result.

\begin{theorem}[Theorem \ref{thm:main3-full}]
\label{thm:main3} 
Let $S$ be a smooth projective surface over $k$
having a decomposition of the diagonal.
Suppose moreover that $H^1_\ur(S)$ is a cyclic group of 
prime order $\ell$.
Then we have
\[
|\CH_0(S_{k(S)})_{\Tor}|=|H^3_\ur(S \times S)|=\ell.
\]
\end{theorem}

This applies to an Enriques surface $S$ (with $\ell=2$),
thereby answering Kahn's question.
(See Example \ref{ex:chow-computation}
for this point and for more examples.)
It also provides us with counter-examples
for the integral Hodge conjecture
(see Corollary \ref{cor:integHC}).

\subsection{A remark on the $\boldsymbol{p}$-part in characteristic $\boldsymbol{p>0}$}
\label{rem:p-in-char-pos}
Suppose now that $k$ has characteristic $p>0$.
As alluded to in the beginning of the introduction, 
most of our proof works over $k$ for the non-$p$-primary torsion part,
with the help of an isomorphism $\Z/m\Z \cong \mu_m$ for $m \in \Z_{>0}$ invertible in $k$.

To pursue a $p$-primary analogue of our arguments,
one may consider a $p$-adic counterpart of the unramified cohomology,
which is defined, for $i, j \in \Z_{\ge 0}$ and a smooth $k$-scheme $X$, as
\[ H^{i, j}_\ur(X)\{p\} := \colim_{n \geq 1} H^0_\Zar(X, \sH^{i,j}_{p^n}). \]
Here $\sH^{i, j}_{p^n}$
 is the Zariski sheaf on $X_\Zar$ associated to the presheaf
 $U \mapsto H^{i-j}_\et(U, W_n \Omega_{U, \log}^j)$,
 and $W_n \Omega_{U, \log}^j$ is the \'etale subsheaf of the logarithmic part of
 the Hodge-Witt sheaf $W_n \Omega_{U}^j$ (see \cite{I}).
The functors $H^{i,j}_\ur(-)\{p\}$ are birational, and motivic by \cite[Proposition 1.3]{KOY} and Proposition \ref{prop:P1inv-bir} below, and normalized for $(i,j) \ne (0,0)$.
However, the groups $H^{i, j}_\ur(S)\{p\}$ do {\it not} necessarily detect the $p$-primary torsion part
 $\CH_0(S_{k(T)})_{p\text{-}\Tor}$.
In fact, when $S$ is a supersingular Enriques surface over $k$ with $\ch(k)=2$,
 we have $H^{i, j}_\ur(S)\{2\}=0$ for all $(i,j) \ne (0,0)$,
 but $\CH_0(S_{k(S)})_{2\text{-}\Tor}$ is non-zero.
We will discuss this example in detail, later in Remark \ref{rem:p-primary}\,(2) below.

\subsection*{Organization of the paper}
\S 2 is a recollection on the Chow motives and birational motives.
We then study
a torsion direct summand of the Chow motive of a surface 
admitting a decomposition of the diagonal  in \S 3.
A key result is Proposition \ref{prop:GG-V2}.
\S 4 is devoted to a preliminary computation 
of cohomology of torsion motive of a surface.
In \S 5, we employ the method of Vishik \cite{Vishik}
to study the motivic cohomology of a torsion motive 
constructed in \S 3.
This result is then applied to deduce an exact sequence in \S 6,
which relates the Chow group
$\CH_0(S_{k(S)})_{\Tor}$ appearing in 
Theorems \ref{thm:main2} and \ref{thm:main3}
with the unramified cohomology.
The main results 
(Theorems \ref{thm:main2-full}, \ref{thm:main1-full}, \ref{thm:main3-full})
are proved in \S 7,
which also contains a discussion of examples and related topics.
\S 8 is an appendix
where we prove elementary results
on homological algebra that are used in the body of the paper.
Another appendix \S 9 contains a proof of the proposition 
saying that 
a $\P^1$-invariant Nisnevich sheaf with transfer 
is a motivic and birational functor.

\subsection*{Acknowledgement} 
We would like to thank Bruno Kahn 
for his invaluable comments for the earlier version,
as well as for his permission to include Proposition \ref{prop:P1inv-bir}.
We are very grateful to the referee for
his or her careful reading and useful comments
that improved the paper very much.

\subsection*{Notations and conventions}
We use the following notations 
throughout this paper. 
\begin{itemize}
\item $k$ is a field,
which will be assumed to be algebraically closed
from \S \ref{sect:surf} onward.
\item
$p$ is the characteristic of $k$ if it is positive,
and $p:=1$ otherwise.
\item
$\Lambda$ is either $\Z, \Z[1/p]$ or $\Q$.
From \S \ref{sect:surf} onward,
we assume $\Lambda=\Z[1/p]$.
\end{itemize}
Notations relative to $k$.
\begin{itemize}
\item 
$\Fld$ is the category of fields over $k$
and $k$-homomorphisms.
Denote by $\Fld^\fg$ (resp. $\Fld^\alg$)
its full subcategory consisting of
those which are finitely generated over $k$
(resp. algebraically closed).
\item 
$\Sch$ is the category of 
separated $k$-schemes of finite type
and $k$-morphisms.
Its full subcategory consisting of 
smooth 
(resp. smooth and projective)
$k$-schemes
is denoted by $\Sm$ (resp. $\SmProj$).
We write $\times$ for the product in $\Sch$
(i.e., 
the fiber product over $\Spec k$ 
in the category of all schemes).
\end{itemize}
Notations relative to $X \in \Sch$.
\begin{itemize}
\item
$X_R := X \times_{\Spec k} \Spec R$ for a $k$-algebra $R$.
\item
$K(X)$ is the total ring of fractions of $X_K$
for $K \in \Fld$.
\item
$X_{(i)}$
is the set of all points of $X$ of dimension $i$ for $i \in \Z$.
\item
$\CH_i(X)$
is the Chow group of dimension $i$ cycles on $X$
for $i \in \Z$.
\item
$\Pic(X)$ 
is the Picard group of $X$.
\item
$\NS(X)$ is the N\'eron-Severi group 
if $X \in \Sm$.
\end{itemize}
Additional general notations,
where $A$ is an abelian group:
\begin{itemize}
\item 
$A[m]:=\{ a \in A \mid ma=0 \}$ for $m \in \Z_{>0}$,
$A_\Tor:=\cup_{m \in \Z_{>0}} A[m]$,
and $A_\fr:=A/A_\Tor$.
\item 
$\exp(A):=\inf\{ m \in \Z_{>0} \mid m A=0 \}
\in \Z_{>0} \cup \{ \infty \}$.
\item
$A_R:=A \otimes_\Z R$ for a commutative ring $R$.
\item 
The set of all morphisms from $X$ to $Y$
in a category $\sC$ is written by $\sC(X, Y)$.
\item 
$\lMod$ is the category of all $\Lambda$-modules
and $\Lambda$-homomorphisms.
\end{itemize}

\section{Preliminaries}

In this section we recall 
some definitions and results 
from \cite{CT, GG, K1, KS1, T, Vishik} 
that will be used later.

\subsection{Chow motives}\label{sect:cat-mot}

We write $\Chow(k)_\Lambda$ for 
the \emph{covariant} category of Chow motives
over $k$ with coefficients in $\Lambda$,
defined e.g. in
\cite[\S 1.5, 1.6]{K1}, \cite[\S 4, p.2092]{T}.
(This is opposite of the more frequently used
contravariant version, see e.g. \cite{S}.)
It is 
a $\Lambda$-linear rigid symmetric monoidal 
pseudo-abelian category.
Any object of $\Chow(k)_\Lambda$
can be written as
$(X, \pi, r)$ for some 
equidimensional $X \in \SmProj$,
a projector $\pi$ of $X$,
and $r \in \Z$.
(By a projector of $X$ we mean 
$\pi \in \CH_{\dim X}(X \times X)_\Lambda$
such that $\pi \circ \pi = \pi$,
where $\circ$ denotes the composition of algebraic correspondences.)
We have
\[ \Chow(k)_\Lambda((X, \pi, r), (Y, \rho, s)) 
= \rho \circ \CH_{\dim X + r - s}(X \times Y)_\Lambda \circ \pi,
\]
where 
$X, Y \in \SmProj$ (with $X$ equidimensional),
$\pi, \rho$ projectors of $X, Y$, and $r, s \in \Z$.
We write
$\Lambda(r):=(\Spec k, \id_{\Spec k}, r)$
and $M(r) := M \otimes \Lambda(r)$
for $M \in \Chow(k)_\Lambda$.
Thus $\Lambda := \Lambda(0)$
is a unit object for the monoidal structure.
We denote by $M^\vee$ the dual object of $M$.

The category of effective Chow motives 
$\Chow(k)^\eff_\Lambda$
is the full subcategory of $\Chow(k)_\Lambda$
consisting of all objects isomorphic to
those of the form $(X, \pi, r)$ with $r \ge 0$.
There is a covariant functor
\begin{equation}\label{eq:h-eff}
h^\eff : \SmProj \to \Chow(k)_\Lambda^\eff,
\qquad h^\eff(X)=(X, \id_X, 0).
\end{equation}
We have $h^\eff(X) = h^\eff(X)^\vee(d)$
if $X \in \SmProj$ is purely $d$-dimensional.
For $M \in \Chow(k)_\Lambda$ and $r \in \Z$ we write
$\CH_r(M)_\Lambda := \Chow(k)_\Lambda(\Lambda(r), M)$
so that we have 
$\CH_r(h^\eff(X))_\Lambda=\CH_r(X)_\Lambda$
for any $X \in \SmProj$.

We abbreviate 
$\Chow_\Lambda:=\Chow(k)_\Lambda$
and
$\Chow^\eff_\Lambda:=\Chow(k)^\eff_\Lambda$.
For any $K \in \Fld$,
there is a base change functor
$\Chow_\Lambda \to \Chow(K)_\Lambda$
written by $M \mapsto M_K$.

\subsection{Torsion motives}

Vishik \cite[Definition 2.4]{Vishik} 
defines a torsion motive to be
an object $M \in \Chow_\Lambda$ such that
$m \cdot \id_M=0$ for some $m \in \Z_{>0}$.
Since we will need a similar notion
considered in different categories,
we introduce the following general terminology:

\begin{definition}\label{def:tor-mot}
We say an object $A$ of an additive category $\sC$ 
is \emph{torsion}
if there exists $m \in \Z_{>0}$ such that
$m \cdot \id_A=0$ in $\sC(A, A)$.
This is equivalent to saying that
$\sC(A, B)$ (or $\sC(B, A)$) is a torsion abelian group
for any $B \in \sC$.
\end{definition}

The following is an obvious variant
of a result of Gorchinskiy-Guletskii \cite[Lemma 1]{GG}
(compare \cite[Proposition 2.1]{Di}).

\begin{lemma}\label{lem:GG}
For $M \in \Chow_\Lambda$,
the following conditions are equivalent.
\begin{enumerate}
\item 
$M$ is a torsion object of $\Chow_\Lambda$.
\item 
$\CH_n(M_K)_\Lambda$ is torsion 
for any $n \in \Z$ and
for any $K \in \Fld$.
\item 
$\CH_n(M_K)_\Lambda$ is torsion 
for any $n \in \Z$ and
for any $K \in \Fld^\alg$.
\item 
$\CH_n(M_K)_\Lambda$ is torsion 
for any $n \in \Z$ and
for any $K \in \Fld^\fg$.
\end{enumerate}
\end{lemma}
\begin{proof}
(2) $\Rightarrow$ (3) and
(2) $\Rightarrow$ (4) are obvious.
(3) $\Rightarrow$ (2)
holds because 
$\ker(\CH_n(M_K)_\Lambda 
\to \CH_n(M_{\ol{K}})_\Lambda)$ is torsion,
where $\ol{K}$ is an algebraic closure of $K \in \Fld$.
(4) $\Rightarrow$ (2) is seen by taking colimit.
We have shown the equivalence 
(2) $\Leftrightarrow$ (3) $\Leftrightarrow$ (4).

Let us show (1) $\Rightarrow$ (4).
By the shown equivalence (3) $\Leftrightarrow$ (4),
we are reduced to the case $k$ is algebraically closed
(in particular $k$ is perfect).
Take $K \in \Fld^\fg$.
By Nagata's compactification 
and de Jong's alteration 
(see \cite[Theorem 4.1]{Co}, \cite[Theorem 4.1]{dJ}),
we can find an integral proper $k$-scheme $X \in \Sch$ with $K=k(X)$
and a proper surjective generically finite morphism $f : Y \to X$
with $Y \in \SmProj$ integral.
We then have a sequence of induced maps
\[
\CH_{n+d_Y}(M \otimes Y)_\Lambda \twoheadrightarrow 
\CH_{n}(M_{k(Y)})_\Lambda \overset{f_*}{\to} \CH_{n}(M_{k(X)})_\Lambda,
\]
where $d_Y:=\dim Y$.
The first map is surjective,
and the cokernel of the second map is annihilated by $[k(Y): k(X)]$.
Since $\CH_{n+d_Y}(M \otimes Y)_\Lambda
=\Chow_\Lambda(\Lambda(n+d_Y), M \otimes Y)$ 
is torsion 
by the assumption (1),
we conclude that $\CH_{n}(M_{k(X)})_\Lambda$ is torsion as well.

It remains to prove
(2) $\Rightarrow$ (1),
for which 
we follow \cite[Lemma 1]{GG}.
Write $M=(X, \pi, r) \in \Chow_\Lambda$ with 
$X$ equidimensional and put $d_X := \dim X$.
We take $N \in \Chow_\Lambda$ and
show that $\Chow_\Lambda(M, N)$ is torsion.
We may assume $N=h^\eff(Y)$ for connected $Y \in \SmProj$
(by replacing $r$ if necessary).
Given $Z \in \Sch$,
we define $\CH_n(M \otimes Z)_\Lambda$ 
as the image of 
an idempotent operator
\[
\CH_n(X \times Z)_\Lambda \to \CH_n(X \times Z)_\Lambda,
\qquad
\alpha \mapsto p_{23 *}(p_{13}^*(\alpha) \cdot_{p_{12}} \pi),
\]
where $p_{ij}$ are respective projections on 
$X \times X \times Z$,
and $\cdot_{p_{12}}$ is the global product along $p_{12}$
defined in \cite[\S 8.1]{F};
this product exists since $X \times X$ is smooth.
We show that $\CH_n(M \otimes Z)_\Lambda$ is torsion 
for any integral $Z \in \Sch$ and for any $n$
by induction on $d_Z:=\dim Z$.
The case $d_Z=0$ is immediate from the assumption (2).
If $d_Z>0$, 
from the localization sequence for $X \times Z$
we deduce an exact sequence
\[
\bigoplus_W \CH_n(M \otimes W)_\Lambda
\to \CH_n(M \otimes Z)_\Lambda 
\to \CH_{n-d_Z}(M_{k(Z)})_\Lambda \to 0,
\]
where $W$ runs through 
integral proper closed subschemes of $Z$.
The claim now follows by induction.
Applying this to $Z=Y$ and $n=d_X+r$, 
we conclude $\CH_{d_X+r}(M \otimes Y)_\Lambda
=\Chow_\Lambda(M, N)$ is torsion.
\end{proof}

\subsection{Birational motives}\label{sect:birat-mot}

We write $\Chow_\Lambda^\bir$ for 
the category of birational motives 
over $k$ with coefficients in $\Lambda$ 
from \cite[Definition 2.3.6]{KS1}.
(This is denoted by 
$\Chow^\circ(k, \Lambda)$ in \cite{KS1}.)
It comes equipped with a  functor
$\Chow_\Lambda^\eff \to \Chow_\Lambda^\bir$.
We write the composition of it with $h^\eff$ by
\begin{equation}\label{eq:h-bir}
h^\bir : \SmProj \to \Chow^\bir_\Lambda.
\end{equation}
We then have 
\[ \Chow_\Lambda^\bir(h^\bir(X), h^\bir(Y)) 
= \CH_0(Y_{k(X)})_\Lambda
\]
for any $X, Y \in \SmProj$
(see \cite[Lemma 2.3.7]{KS1}).

\begin{remark}\label{rem:comparison}
There are several variants of $\Chow^\bir_\Lambda$.
We recall two of them.
\begin{enumerate}
\item 
Denote by 
$\Chow_\Lambda^{\bir, 1}$
the pseudo-abelian envelope of the 
category obtained from 
$\Chow^\eff_\Lambda$
by inverting all birational morphisms.
\item 
Denote by 
$\Chow_\Lambda^{\bir, 2}$
the pseudo-abelian envelope
of $\Chow^\eff_\Lambda/\L$,
where 
$\L$ is the ideal of $\Chow^\eff_\Lambda$
consisting of all morphisms which 
factor through an object of the form $M(1)$
with $M \in \Chow^\eff_\Lambda$.
\end{enumerate}
There are functors
\[
\Chow_\Lambda^{\bir, 2}
\overset{\cong}{\longrightarrow}
\Chow_\Lambda^{\bir, 1}
\overset{}{\longrightarrow}
\Chow_\Lambda^{\bir}.
\]
The first one is always an equivalence,
and so is the second 
at least if $p$ is invertible in $\Lambda$
(see \cite[Proposition 2.2.9, Corollary 2.4.3]{KS1}).
As 
$\Chow^\eff_\Lambda \to \Chow^\bir_\Lambda$
factors through $\Chow_\Lambda^{\bir, 2}$,
the image of $M(1)$ vanishes in $\Chow^\bir_\Lambda$
for any $M\in \Chow^\eff_\Lambda$.
\end{remark}

Finally, we write $\Chow^\nor_\Lambda$ for 
the quotient category of $\Chow^\bir_\Lambda$
by the ideal consisting of all morphisms which 
factor through $\Lambda=h^\bir(\Spec k)$,
introduced in \cite[Definition 2.4]{K1}.
Denote by 
\begin{equation}\label{eq:h-nor}
h^\nor : \SmProj \to \Chow^\nor_\Lambda
\end{equation}
the composition of $h^\bir$ and 
the localization functor
$\Chow^\bir_\Lambda \to \Chow^\nor_\Lambda$.
We have 
\begin{equation} \label{eq:hom-in-chow-nor}
\Chow^\nor_\Lambda(h^\nor(X), h^\nor(Y)) = 
\coker(\CH_0(Y)_\Lambda \to \CH_0(Y_{k(X)})_\Lambda)
\end{equation}
for any $X, Y \in \SmProj$ 
(see loc. cit.).

\begin{remark}\label{rem:abbre-mot}
If no confusion is likely,
we abbreviate 
$h^\eff(X),~ h^\bir(X)$, and $h^\nor(X)$
by $X$ for $X \in \SmProj$.
Similarly,
for $M \in \Chow^\eff_\Lambda$
we use the same letter $M$ to denote
its images in 
$\Chow^\bir_\Lambda$ and $\Chow^\nor_\Lambda$.
For instance,
the left hand side of \eqref{eq:hom-in-chow-nor}
will be written by
$\Chow^\nor_\Lambda(X, Y)$.
\end{remark}

\subsection{Motivic invariants}\label{sect:mot-inv}

Denote by $\lMod$ the category of
$\Lambda$-modules.
Following \cite[Definition 2.1]{K1},
we introduce some definitions.

\begin{definition}\label{def:bir-mot-inv}
Let $F : \SmProj^\op \to \lMod$ be a functor.
\begin{enumerate}
\item 
We say $F$ is \emph{birational}
if $F(f)$ is an isomorphism for any
birational morphism $f$.
\item
We say $F$ is \emph{motivic}
if $F$ factors through 
an additive functor $\Chow^{\eff, \op}_\Lambda \to \lMod$.
\item
We say $F$ is \emph{normalized} if
$F(\Spec k)=0$.
\end{enumerate}
\end{definition}

\begin{lemma}\label{lem:bir-funct-factor}
Suppose that $p$ is invertible in $\Lambda$.
A functor $F : \SmProj^\op \to \lMod$
is birational and motivic 
(resp.  normalized, birational, and motivic)
if and only if $F$ factors through 
an additive functor 
$\Chow^{\bir, \op}_\Lambda \to \lMod$
(resp. $\Chow^{\nor, \op}_\Lambda \to \lMod$).
\end{lemma}
\begin{proof}
This is immediate from what we recalled in 
\S \ref{sect:birat-mot}.
\end{proof}

\begin{remark}\label{rem:extension-F}
Given a motivic 
(resp. birational and motivic,
resp. normalized, birational, and motivic)
functor $F : \SmProj^\op \to \lMod$,
its extension to 
$\Chow^\eff_\Lambda$
(resp. $\Chow^\bir_\Lambda$, resp. $\Chow^\nor_\Lambda$)
is denoted by the same letter $F$.
\end{remark}

\begin{example}\label{ex:bir-mot-ft}
\begin{enumerate}
\item 
Suppose $p=1$ or $\Lambda=\Z$.
It is a classical fact that
$H^0(-, \Omega^i_{-/k})$ 
is birational and motivic 
for any $i \in \Z_{\ge0}$;
it is also normalized if $i>0$.
It is less classical that the same is true of
$H^i(-, \sO)$ if $k$ is perfect
(see \cite{CR}).
\item 
It is obvious from the definition that
the functor
\begin{equation} \label{eq:CH0-S}
\Chow^\nor_\Lambda(-, S) :
T \mapsto \Chow^\nor_\Lambda(T, S)
= 
\Coker(\CH_0(S)_\Lambda \to \CH_0(S_{k(T)})_\Lambda)
\end{equation}
is birational, motivic, and normalized
for any fixed $S \in \SmProj$.
\item 
Let $M$ be a cycle module in the sense of Rost \cite{R}.
Then its $0$-th cycle cohomology $A^0(-, M_n)$ 
is birational and motivic
by \cite[Corollary 6.1.3]{KS1}.
We will only use a special case of unramified cohomology,
which will be recalled in the next subsection.
\item
A $\P^1$-invariant Nisnevich sheaf with transfers
is birational and motivic.
We include a proof of this fact,
due to Bruno Kahn, in an appendix
(see Proposition \ref{prop:P1inv-bir} below).
This recovers all examples discussed above,
except $H^i(-, \sO)$.
\end{enumerate}
\end{example}


\subsection{Unramified cohomology}\label{sect:unram-coh}

A general reference for this subsection is \cite{CT}.
Let $K \in \Fld$ and $i \in \Z$.
For $n \in \Z_{>0}$ invertible in $k$,
the \emph{unramified cohomology} of $K/k$ 
is defined by
\begin{equation}\label{eq:def-unramcoh1}
H_{\ur, n}^i(K/k) := \ker\left(
H_\Gal^i(K, \mu_n^{\otimes (i-1)}) \to \bigoplus_v 
H_\Gal^{i-1}(F_v, \mu_n^{\otimes (i-2)}) \right),
\end{equation}
where $v$ ranges over all discrete valuations of $K$ 
that are trivial on $k$,
and $F_v$ is the residue field of $v$.
The maps appearing in the definition are the residue maps
(see \cite[(3.6)]{CT}).
We set
\begin{equation}\label{eq:def-unramcoh2}
H_\ur^i(K/k) := \colim_{(n, p)=1} H_{\ur, n}^i(K/k),
\end{equation}
where $n$ ranges over all $n \in \Z_{>0}$ 
that is invertible in $k$.
By Rost-Voevodsky's norm residue isomorphism theorem 
(which is the former Bloch-Kato conjecture
and proved in \cite[Theorem 6.16]{V-lcoef}), 
we may identify $H_{\ur, n}^i(K/k)$
with the $n$-torsion part of $H_\ur^i(K/k)$:
\begin{equation}\label{eq:def-unramcoh3}
H_{\ur, n}^i(K/k) \cong 
H_\ur^i(K/k)[n].
\end{equation}

Let $X \in \Sm$ and $i \in \Z$.
For $n \in \Z_{>0}$ invertible in $k$, 
the \emph{unramified cohomology} of $X$ 
is defined as
\begin{equation}\label{eq:def-unramcoh4}
H^i_{\ur, n}(X)
:=H^0_\Zar(X, \sH^i_n),
\qquad
H^i_\ur(X) := \colim_{(n, p)=1} H^i_{\ur, n}(X),
\end{equation}
where 
$\sH^i_n$ is the Zariski sheaf on $X$
associated to the presheaf
$U \mapsto H^i_\et(U, \mu_n^{\otimes (i-1)})$,
and the colimit in the second formula is taken 
in the same way as \eqref{eq:def-unramcoh2}.
We have canonical isomorphisms
(see \cite[Propositions 4.2.1, 4.2.3]{CT})
\begin{equation} \label{eq:ur-coh-12}
H^1_{\ur, n}(X) \cong H^1_\et(X, \Z/n\Z), 
\qquad
  H^2_{\ur, n}(X) \cong \Br(X)[n],
\end{equation}
where $\Br(X):=H^2_\et(X, \G_m)$ is the Brauer group of $X$.
If further $X$ is integral and proper over $k$,
we also have
(see \cite[Theorem 4.1.1]{CT})
\begin{equation} \label{eq:unramcoh-K-S}
H^i_{\ur, n}(X) \cong H^i_{\ur, n}(k(X)/k),
\qquad 
H^i_{\ur}(X) \cong H^i_{\ur}(k(X)/k).
\end{equation}

The following well-known fact 
plays an essential role in this paper:

\begin{proposition}\label{prop:unram-coh}
Let $i, n \in \Z$ and suppose that $n$ is invertible in $k$.
Then the functor
$H_{\ur, n}^i : \SmProj \to \Mod_{\Z[1/p]}$
is birational and motivic.
The same is true for $H_\ur^i$.
They are also normalized 
if $i>0$ and $k$ is algebraically closed.
\end{proposition}
\begin{proof}
The first statement follows from
\cite[Theorem 4.1.1]{CT}
(see also \cite[(2.5)]{R}) and \cite[Corollary 6.1.3]{KS1},
and the second from the first.
The third statement is obvious from the definition.
\end{proof}

\subsection{Varieties admitting a decomposition of the diagonal}\label{sect:dec-diag}

\begin{proposition}\label{prop:equiv-UTC}
The following conditions are equivalent
for $X \in \SmProj$:
\begin{enumerate}
\item[(1)] 
The degree map induces an isomorphism
$\CH_0(X_{k(X)})_\Q \cong \Q$.
\item[(2)]
The class of the generic point of $X$
in $\CH_0(X_{k(X)})_\Q$
belongs to 
\[ \Im(\CH_0(X)_\Q \to \CH_0(X_{k(X)})_\Q).
\]
\item[(3)] 
The structure map induces an isomorphism
$h^\bir(X) \cong \Q$ 
in $\Chow^\bir_\Q$ 
\item[(4)] 
The object $h^\nor(X)$ of $\Chow^\nor_\Z$ 
is torsion in the sense of Definition \ref{def:tor-mot}.
\end{enumerate}
\end{proposition}
\begin{proof}
See \cite[Proposition 3.1.1]{KS1} for (1)--(3).
Equivalence of (2) and (4) is obvious
from the definition and \eqref{eq:hom-in-chow-nor}
(see also \cite[\S 2.3]{K1}).
\end{proof}

\begin{remark}
If $k$ is an algebraically closed field
with infinite transcendental degree
over its prime subfield,
then these conditions 
are also equivalent to the following:
\begin{enumerate}
\item[(1)']
The degree map induces an isomorphism
$\CH_0(X)_\Lambda \cong \Lambda$
for either $\Lambda=\Z$ or $\Q$.
\end{enumerate}
(See \cite[Proposition 3.1.1]{KS1}.)
\end{remark}

\begin{definition}\label{def:dec-diag}
We say $X \in \SmProj$ 
admits a
\emph{decomposition of the diagonal}
if the conditions of
Proposition \ref{prop:equiv-UTC} are satisfied.
\end{definition}

This notion goes back to Bloch-Srinivas \cite{BS}.
For such $X$,
Kahn \cite[Definition 2.5]{K1}
and
Chatzistamatiou-Levine \cite[Definition 1.1]{CL}
defined a numerical invariant called 
the \emph{torsion order},
which can be written as
$\Tor_\Z^\nor(X)$
in terms of the following definition:

\begin{definition}\label{def:tor-ord}
\begin{enumerate}
\item 
Let $A$ be an object of an additive category $\sC$
that is  torsion in the sense of Definition 
\ref{def:tor-mot}.
The smallest $m \in \Z_{>0}$ such that
$m \cdot \id_A=0$ 
is called  the \emph{torsion order} of $A$.
\item
The torsion order of 
a torsion object $M$ of $\Chow_\Lambda^\eff$
(resp. $\Chow_\Lambda^\bir$, resp. $\Chow_\Lambda^\nor$)
is denoted by 
$\Tor^\eff_\Lambda(M)$
(resp. $\Tor^\bir_\Lambda(M)$,
resp. $\Tor^\nor_\Lambda(M)$).
\end{enumerate}
\end{definition}

We write $b_i(X)$ and $\rho(X)$
for the Betti and Picard numbers of $X \in \SmProj$:
\[ b_i(X) := \dim_{\Q_\ell} H^i_\et(X_{\ol{k}}, \Q_\ell),
\qquad 
   \rho(X):= \rank_\Z \NS(X_{\ol{k}})/\NS(X_{\ol{k}})_\Tor,
\]
where $\ol{k}$ is an algebraic closure of $k$,
and $\ell$ is any prime number different from $p$.

\begin{proposition}\label{prop:b1-b2-rho}
Suppose that 
$X \in \SmProj$ admits
a decomposition of the diagonal.
\begin{enumerate}
\item 
We have $b_1(X)=0$, ~$b_2(X)=\rho(X)$
and $\Pic(X)=\NS(X)$.
\item 
Suppose that $k$ is algebraically closed.
For any prime number $\ell$ invertible in $k$,
we have canonical isomorphisms
\begin{align*}
&H^1_\et(X, \Q_\ell/\Z_\ell(1)) \cong \NS(X)_{\Tor, \Z_\ell},
\\
&H^1_\ur(X)_{\Z_\ell} \cong H^1_\et(X, \Q_\ell/\Z_\ell) 
\cong H^2_\et(X, \Z_\ell)_\Tor,
\\
&H^2_\ur(X)_{\Z_\ell} \cong \Br(X)_{\Z_\ell}
\cong H^3_\et(X, \Z_\ell(1))_\Tor.
\end{align*}
\item 
Suppose that $p$ is invertible in $\Lambda$,
and put  $m:=\Tor^\nor_\Lambda(X)$.
Then we have $m  F(X)=0$
for any normalized, birational, and motivic 
functor $F : \SmProj^\op \to \lMod$.
\end{enumerate}
\end{proposition}
\begin{proof}
See \cite[Proposition 3.1.4]{KS1} for the proof of (1)
and \cite[Lemma 2.6]{K1} for (3).
(2) follows from (1), \eqref{eq:ur-coh-12} 
and the following Lemma.
\end{proof}

\begin{lemma}\label{lem:H23tor-Hur-surj}
Suppose that $k$ is algebraically closed.
Let $\ell$ be a prime number invertible in $k$.
For any $X \in \SmProj$,
we have a canonical isomorphism
\begin{equation}\label{eq:coh1}
H^1_\et(X, \Q_\ell/\Z_\ell(1)) \cong \Pic(X)_{\Tor, \Z_\ell}
\end{equation}
and canonical surjective morphisms
\begin{equation}\label{eq:coh2}
H^1_\et(X, \Q_\ell/\Z_\ell) 
\twoheadrightarrow H^2_\et(X, \Z_\ell)_\Tor,
\qquad
\Br(X)_{\Z_\ell} \twoheadrightarrow H^3_\et(X, \Z_\ell(1))_\Tor.
\end{equation}
Moreover, the first (resp. second)
morphism in \eqref{eq:coh2}
is bijective if $b_1(X)=0$
(resp. $b_2(X)=\rho(X)$).
\end{lemma}
\begin{proof}
For any $m, n \in \Z$ with $m, n>0$,
we have exact sequences of \'etale sheaves:
\begin{align*}
&
0 \to
\mu_{\ell^m} \to
\mu_{\ell^{m+n}} \to
\mu_{\ell^n} \to
0,
&
&
0 \to
\mu_{\ell^n} \to
\G_m \to
\G_m \to
0.
\end{align*}
From the second sequence we obtain an isomorphism
$H^1_\et(X, \mu_{\ell^n}) \cong \Pic(X)[\ell^n]$,
from which we deduce \eqref{eq:coh1} by taking a colimit over $n$.
The upper exact row in the following diagram is obtained in a similar way,
while the lower row is obtained by taking a limit over $m$ and a colimit over $n$
of the long exact sequence deduced from the first sequence:
\[
\xymatrix{
0 \ar[r] &
\Pic(X) \otimes \Q_\ell/\Z_\ell \ar[r]  \ar[d]&
H^2_\et(X, \Q_\ell/\Z_\ell(1)) \ar[r]  \ar@{=}[d]&
\Br(X)_{\Z_\ell} \ar[r]  \ar[d]&
0
\\
0 \ar[r] &
H^2_\et(X, \Z_\ell(1)) \otimes \Q_\ell/\Z_\ell \ar[r] &
H^2_\et(X, \Q_\ell/\Z_\ell(1)) \ar[r] &
H^3_\et(X, \Z_\ell(1))_\Tor \ar[r]  &
0.
}
\]
(The limit preserves the exactness of the lower low
since $H^i_\et(X, \mu_{\ell^m})$ is finite
for each $i, m$.)
The left and the right vertical maps are induced
since the composition $\Pic(X) \otimes \Q/\Z \to H^3_\et(X, \Z_\ell(1))_\Tor$ vanishes
(as the source is divisible and the target is finite).
The second surjection in \eqref{eq:coh2}
is obtained as the right vertical map in this diagram,
which is bijective if $b_2(X)=\rho(X)$
because so is the left vertical map under this hypothesis.

By a similar argument with different Tate twist,
we get an exact sequence
\begin{equation}\label{eq:kummer-ex-seq-etcoh}
0 \to
H^i_\et(X, \Z_\ell(r)) \otimes \Q_\ell/\Z_\ell \to
H^i_\et(X, \Q_\ell/\Z_\ell(r)) \to
H^{i+1}_\et(X, \Z_\ell(r))_\Tor \to 0
\end{equation}
for any $i, r \in \Z$.
The first surjection in \eqref{eq:coh2}
is obtained as the second arrow in this sequence
for $(i, r)=(1, 0)$,
which is bijective if $b_1(X)=0$
because the first term vanishes under this hypothesis.
(We will use \eqref{eq:kummer-ex-seq-etcoh} for other $(i, r)$ later.)
\end{proof}

\begin{remark}\label{rem:div-nor-eff}
\begin{enumerate}
\item 
If $S \in \SmProj$ is a surface
such that  $b_1(S)=0$ and $b_2(S)=\rho(S)$,
then Bloch's conjecture predicts that
$S$ should admit a decomposition of the diagonal
(see \cite[Proposition 3.1.4]{KS1}).
\item
It is obvious that
$\Tor^\nor_\Lambda(M) ~|~ 
\Tor^\bir_\Lambda(M) ~|~ 
\Tor^\eff_\Lambda(M)$
for torsion $M \in \Chow^\eff_\Lambda$.
The opposite divisibility 
does not hold in general.
(For example, we have 
$\Tor^\eff_\Lambda(M)=\Tor^\eff_\Lambda(M(1))$
but the image of $M(1)$ vanishes 
in $\Chow^\bir_\Lambda$.)
Yet, it can hold in some non-trivial cases,
as seen in Proposition \ref{prop:GG-V2} below.
\end{enumerate}
\end{remark}

\section{Torsion motives of surfaces}\label{sect:surf}

\begin{setting}\label{setting:surf}
From now on we suppose $k$ is algebraically closed
and $\Lambda=\Z[1/p]$.
Fix $S \in \SmProj$ 
admitting  a decomposition of the diagonal
and such that $\dim S=2$.
\end{setting}

\subsection{Surfaces admitting a decomposition of the diagonal}\label{sect:triv-bir-mot}


\begin{lemma}\label{lem:coh-surf}
For any prime number $\ell \not=p$, we have the following:
\begin{enumerate}
\item 
$b_0(S)=b_4(S)=1$, 
$b_2(S)=\rho(S)$,
and  $b_i(S)=0$ for any $i \not= 0, 2, 4$.
\item 
$H^0_\et(S, \Z_\ell)
=H^4_\et(S, \Z_\ell(2))=\Z_\ell$,~
$H^1_\et(S, \Z_\ell)=0$, and
$H^3_\et(S, \Z_\ell(1))$ is finite.
%
\item 
$\Pic(S)=\NS(S)$
is a finitely generated $\Z$-module;
$\NS(S)_{\Tor, \Lambda}$
and $\Br(S)_{\Lambda}$
are finite abelian groups 
canonically dual to each other.
\item 
$\CH_1(S_K) \cong \NS(S)$
for any $K \in \Fld$
and
$\CH_0(S_{\ol{K}}) \cong \Z$
for any $\ol{K} \in \Fld^\alg$.
\end{enumerate}
\end{lemma}
\begin{proof}
(1) 
Proposition \ref{prop:b1-b2-rho} shows the statement for $i \le 2$.
Then the Poincar\'e duality $b_{4-i}(X)=b_i(X)$ completes the proof for other $i$.

(2) 
All assertions follow from (1),
plus a fact $H^1_\et(S, \Z_\ell)_\Tor=0$
which is seen from \eqref{eq:kummer-ex-seq-etcoh}.

(3) 
Proposition \ref{prop:b1-b2-rho} shows
the first statement.
It also shows 
$\NS(S)_{\Tor, \Z_\ell} \cong H^2_\et(S, \Z_\ell(1))_\Tor$
and $\Br(S)_{\Lambda} \cong H^3_\et(S, \Z_\ell(1))_\Tor$,
hence they are dual to each other by the Poincar\'e duality.

(4)
Proposition \ref{prop:b1-b2-rho} shows 
the vanishing of the Picard variety of $X$,
whence the first statement.
Since this implies the vanishing of the Albanese variety $\Alb_S$ of $S$,
the last statement of (4) follows from Roitman's theorem
\cite[p. 565, Consequence III]{Roit}
(which says $\CH_0(S_{\ol{K}})[m] \cong \Alb_S(\ol{K})[m]$
for any $m \in \Z$ invertible in $k$).
%
%
%
%
%
\end{proof}

\begin{lemma}\label{lem:delta}
Let $\rho:=\rho(S)$
and take $e_1, \dots, e_\rho \in \NS(S)$
such that their classes form a $\Z$-basis of 
$\NS(S)/\NS(S)_\Tor$.
Let $a_{ij} := \langle e_i, e_j \rangle \in \Z$,
where $\langle \cdot, \cdot \rangle$
denotes the intersection form on $S$.
Then $\delta:=\det((a_{ij})_{i,j=1, \dots, \rho})$ is invertible in $\Lambda$.
\end{lemma}
\begin{proof}
It suffices to show that $\delta \in \Z_\ell^\times$
for any prime number $\ell \not= p$.
By Proposition \ref{prop:b1-b2-rho}  
we have an isomorphism
$\NS(S)_{\Z_\ell} \cong H^2_\et(S, \Z_\ell(1))$
which is compatible with the intersection pairing
and the cup product.
Therefore it suffices to show that
the cup product induces an isomorphism
\[
H^2_\et(S, \Z_\ell(1))_\fr
\overset{\cong}{\longrightarrow}
\Hom_{\Z_\ell}(H^2_\et(S, \Z_\ell(1))_\fr, \Z_\ell),
\]
where we put $M_\fr := M/M_\Tor$ for a $\Z_\ell$-module $M$.
This follows \cite[Corollary 1.3]{Zarhin}.
%
%
%
\end{proof}

\begin{proposition}\label{prop:GG-V}
There exists a direct sum decomposition
$h^\eff(S) \cong L \oplus M \oplus N$ 
in $\Chow^\eff_\Lambda$ 
satisfying the following conditions:
\begin{enumerate}
\item 
We have isomorphisms $L \cong \Lambda \oplus \Lambda(2)$
and 
$N \cong  \Lambda(1)^{\rho(S)}$;
\item
$M$ is torsion in $\Chow_\Lambda^\eff$
in the sense of Definition \ref{def:tor-mot};
\item 
We have isomorphisms $L \cong L^\vee(2)$,
$M \cong M^\vee(2)$, and
$N \cong N^\vee(2)$
which are compatible with those in (1) 
and the Poincar\'e duality
$h^\eff(S) \cong h^\eff(S)^\vee(2)$.
\end{enumerate}
\end{proposition}

\begin{proof}
The statement without the condition (3)
is shown by Gorchinskiy-Orlov 
in (the proof of)
\cite[Proposition 2.3, Remark 2.5]{GO} when $k=\C$,
and 
the full statement by Vishik
in \cite[Proposition 4.1]{Vishik}
when $S$ is the classical Godeaux surface.
The same proof works without any essential change,
but for the sake of completeness 
we give a brief account. 

Let $\rho:=\rho(S)$
and take $e_1, \dots, e_\rho \in \NS(S)$
such that their classes form a $\Z$-basis of 
$\NS(S)/\NS(S)_\Tor$.
Let $a_{ij}$ be as in Lemma \ref{lem:delta},
and set
$A:=(a_{ij}) \in \GL_\rho(\Lambda)$.
Write $A^{-1}=(b_{ij}) \in \GL_\rho(\Lambda)$.
Take also a closed point $x_0 \in S_{(0)}$.
We then define orthogonal projectors
\begin{align*}
& 
\pi_L := [S \times x_0] + [x_0 \times S],
\quad
\pi_N := \sum_{i, j} b_{ij} [e_i \times e_j]
\quad
\in \Chow_\Lambda^\eff(S, S)=\CH_2(S \times S)_\Lambda.
\end{align*}
Set 
$L:=(S, \pi_L, 0), N:=(S, \pi_N, 0),  
M:=(S, 1-\pi_L-\pi_N) \in \Chow^\eff_\Lambda$.
Then we have (1) and (3).
Observe that (1) and Lemma \ref{lem:coh-surf} imply that
for any $\ol{K} \in \Fld^\alg$
\begin{equation}
\CH_1(M_{\ol{K}})_\Lambda=\NS(S)_{\Lambda, \Tor}
\quad \text{and}\quad 
\CH_i(M_{\ol{K}})_\Lambda=0 ~\text{for}~ i \not= 1.
\end{equation}
It then follows by Lemma \ref{lem:GG} that
$M$ satisfies (2) too.
We are done.
%
\end{proof}

The summand $M$ is not necessarily unique.
We choose one and fix it.

\begin{setting}\label{setting:def-M}
In what follows we denote by 
$M \in \Chow^\eff_\Lambda$ 
a Chow motive constructed in Proposition \ref{prop:GG-V}.
Observe that we have 
$S=M$ in $\Chow^\nor_\Lambda$,
because 
$\Lambda(r)$ vanishes in $\Chow^\nor_\Lambda$ for any $r \ge 0$
by Remark \ref{rem:comparison}.
\end{setting}

\subsection{Injectivity}

The following proposition proves the injectivity 
of the first map in \eqref{eq:Vishik-ex-seq}.

\begin{proposition}\label{prop:GG-V2}
\begin{enumerate}
\item 
We take $T \in \SmProj$ 
and consider the maps
\[
\Chow^\eff_\Lambda(T, M) 
\overset{a}{\longrightarrow} 
\Chow^\nor_\Lambda(T, M)
\overset{b}{\longrightarrow} 
\underset{i=1, 2}{\bigoplus} \Hom(H_\ur^i(M), H_\ur^i(T)),
\]
where $a$ is induced by the functor
$\Chow_\Lambda^\eff \to \Chow_\Lambda^\nor$,
and $b$ is induced by the functors $H^i_\ur$ for $i=1, 2$
using Lemma \ref{lem:bir-funct-factor} and Proposition \ref{prop:unram-coh}.
Then $a$ is bijective and $b$ is injective.
\item
We have 
\begin{equation}\label{eq:torsion-order-S}
\Tor_\Lambda^\eff(M) 
= \Tor_\Lambda^\nor(M) 
= \Tor_\Lambda^\nor(S) 
= \exp(\NS(S)_{\Tor, \Lambda})
= \exp(\Br(S)_\Lambda),
\end{equation}
where $\exp(A):= \min \{ m \in \Z_{>0} ~|~ m A = 0 \}$
for an abelian group $A$.
\end{enumerate}
\end{proposition}

\begin{proof}
(1) 
(Compare \cite[Proposition 2.3]{GO}.)
We consider a commutative diagram
\[
\xymatrix{
\Chow^\eff_\Lambda(T, M) \ar@{->>}[r]^-a \ar[rd]^e &
\Chow^\nor_\Lambda(T, M) \ar[d]^b
\\
\Chow^\eff_\Lambda(T, S) \ar@{->>}[u]^c \ar[r]^-d &
\underset{i=1, 2}{\bigoplus} \Hom(H_\ur^i(M), H_\ur^i(T)). 
}
\]
The maps $a$ and $c$ are surjective by definition.
Therefore it suffices to prove the injectivity of $e$.
Take $f \in \Chow_\Lambda^\eff(T, M)$ such that $e(f)=0$.
By 
Proposition \ref{prop:b1-b2-rho} (2) and
Lemma \ref{lem:H23tor-Hur-surj}, 
this implies that,
for any prime number $\ell \not= p$, we have
\begin{equation}\label{eq:vanishing-23}
f^*=0 : H^i_\et(M, \Z_\ell(1))_{\Tor} 
\to H^i_\et(T, \Z_\ell(1))_{\Tor}
\quad \text{for $i=2, 3$}.
\end{equation}

On the other hand,
we have a commutative diagram
\[
\xymatrix{
\CH^2(M \otimes T)_{\Tor, \Z_{\ell}} 
\ar@{^{(}-^{>}}[r] \ar[rd]_{\cyc}
&
H^3_\et(M \otimes T, \Q_\ell/\Z_\ell(2)) \ar[d]^\cong 
\\
\Chow^\eff_\Lambda(T, M)_{\Z_\ell} \ar@{=}[u]  \ar[r]
&
H^4_\et(M \otimes T, \Z_\ell(2))_\Tor. 
}
\]
Here $\cyc$ is the cycle map.
The upper horizontal injective map is 
the one constructed by Bloch
(see \cite[Th\'eor\`eme 4.3]{CT2}).
The upper right triangle is commutative
by \cite[Corollaire 4]{CSS}.
The right vertical map is bijective since
we have $H^*_\et(M \otimes T, \Q_\ell(2))=0$
(as $M$ is torsion).
We have shown the injectivity of $\cyc$. 
We consider isomorphisms
\begin{align*}
H^4_\et(M \otimes T, \Z_\ell(2))_{\Tor}
&\cong 
\underset{i=2, 3}{\bigoplus} 
\Tor(H^{5-i}_\et(M, \Z_\ell(1))_{\Tor}, H^i_\et(T, \Z_\ell(1))_{\Tor}) 
\\
&\cong
\underset{i=2, 3}{\bigoplus} 
\Hom(H^i_\et(M, \Z_\ell(1))_{\Tor}, 
H^i_\et(T, \Z_\ell(1))_{\Tor})
\end{align*}
induced by the K\"unneth formula, 
Poincar\'e duality (together with Proposition \ref{prop:GG-V} (3)),
and Lemma \ref{lem:homalg} below.
Their composition sends $\alpha$ to the correspondence action
(that is, 
$\beta \mapsto \pr_{2*}(\pr_1^*(\beta) \cup \alpha)$
where $\pr_i$ are projections on $M \otimes T$).
Hence it fits in 
the right vertical arrow of a commutative diagram
\[
\xymatrix{
\CH^2(M \otimes T)_{\Tor, \Z_{\ell}} \ar@{^{(}-^{>}}[r]^{\cyc}
&
H^4_\et(M \otimes T, \Z_\ell(2))_\Tor \ar[d]^\cong
\\
\Chow^\eff_\Lambda(T, M)_{\Z_\ell}
\ar[r] \ar@{=}[u]   
&
\underset{i=2, 3}{\bigoplus} 
\Hom(H^i_\et(M, \Z_\ell(1))_{\Tor}, 
H^i_\et(T, \Z_\ell(1))_{\Tor}),
}
\]
where the lower horizontal map is induced by
the functors $H^i_\et(-, \Z_\ell(1))_\Tor$ for $i=2, 3$.
Now \eqref{eq:vanishing-23}
shows that $f=0$ in 
$\Chow^\eff_\Lambda(T, M)_{\Z_\ell}$.
We are done.

(2)
The relations
\[
\exp(\NS(S)_{\Tor, \Lambda})= 
\exp(\Br(S)_\Lambda) ~|~
\Tor_\Lambda^\nor(S) =
\Tor_\Lambda^\nor(M) ~|~
\Tor_\Lambda^\eff(M)
\]
are seen by 
Lemma \ref{lem:coh-surf} (3),
Propositions
 \ref{prop:unram-coh} and \ref{prop:b1-b2-rho} (3)
applied to $F=\Br(-)_\Lambda$,
the equality $S=M$ in $\Chow^\nor_\Lambda$, 
and Remark \ref{rem:div-nor-eff} (2), respectively.
To conclude
it suffices to apply (1) to $T=S$ and 
$f=m \cdot \id_S$ with $m \in \Z_{>0}$
to get 
$\Tor_\Lambda^\eff(M) ~|~ \exp(\NS(S)_{\Tor, \Lambda})$.
\end{proof}

We record the following corollary for later use.

\begin{corollary}\label{cor:coh-S-M}
\begin{enumerate}
\item 
If $F : \SmProj^\op \to \lMod$ is a motivic functor,
then $F(M)$ is annihilated by 
the integer in \eqref{eq:torsion-order-S}.
(We used the convention of Remark \ref{rem:extension-F}.)
\item 
We have
$H^i_\et(M, \Z_\ell) 
\cong H^i_\et(S, \Z_\ell)_\Tor$
for any $i \in \Z$ and any prime $\ell \not= p$.
\end{enumerate}
\end{corollary}
\begin{proof}
(1) and (2) follows from
Propositions \ref{prop:GG-V2} and \ref{prop:GG-V} respectively.
\end{proof}


\begin{problem}\label{prob:full-faith}
Let $\sC$ be the full subcategory of $\Chow^\eff_\Lambda$
consisting of torsion direct summands
of the motives of surfaces (not necessarily admitting a decomposition of the diagonal).
Is the functor $\sC \to \Chow^\nor_\Lambda$
fully faithful?
\end{problem}

We end this section with two remarks concerning the $p$-adic counterpart of our results.

\begin{remark}\label{rem:p-primary}
Assume that $p>0$, and let $S$ be as before.
\begin{enumerate}
\item[(1)]
The number $\delta$ for $S$ in Lemma \ref{lem:delta} is {\it not} necessarily
invertible in $\Z$.
For example, when $S$ is a unirational (hence supersingular) K3 surface,
$S$ admits a decomposition of the diagonal, and we have $\delta = -p^{2\sigma_0}$ for some $1 \leq \sigma_0 \leq 10$, cf.\ \cite[Chapter II, \S7.2]{I}.
This example also shows that the decomposition of motives in Proposition \ref{prop:GG-V} {\it does not} hold integrally, in general.
\item[(2)]
Assume further that $\delta$ for $S$ in Lemma \ref{lem:delta} is invertible in $\Z$; 
this is the case for an Enriques surface \cite[Chapter II, Corollary 7.3.7]{I}.
Under this assumption, one can take a torsion motive $M$ of $S$ in $\Chow^\eff_{\Z}$, and consider the canonical homomorphism
\[ b_p : \Chow^\nor_{\Z_p}(T,M) \longrightarrow
 \underset{i,j \geqq 0}{\bigoplus} \ \Hom(H_\ur^{i,j}(M)\{p\}, H_\ur^{i,j}(T)\{p\}). \]
Here $H^{i,j}_\ur(-)\{p\}$ ($i,j \geq 0$) is as in \S \ref{rem:p-in-char-pos}, which is birational and motivic, and normalized for $(i,j) \ne (0,0)$.
However, the map $b_p$ is {\it not} injective in general, even when $T=S$.
We explain this claim in what follows.
First note that $H^{i,j}_\ur(X)\{p\}$ is zero unless $(i,j)=(0,0),(1,0),(1,1),(2,1),(2,2)$ for any surface $X \in \SmProj$; see \cite[Lemma 2.1]{Su} for the vanishing of $H^{3,2}_\ur(X)\{p\}$.
For the torsion motive $M$, we have
$H^{i,j}_\ur(M)\{p\}=0$ unless $(i,j)=(1,0),(1,1),(2,1),(2,2)$.
Noting that $H^{i,j}_\ur(M)\{p\}$ is killed by $\Tor_{\Z_p}^\eff(M)$, we have
\[ H^{i,j}_\ur(M)\{p\} \cong
 \varinjlim_{n \geq 1} \ H^{i-j}_{\et}(M,W_n\Omega^j_{S,\log})
 \cong H^{i-j+1}_{\et}(S,W\Omega^j_{S,\log})_{\Tor}, \]
where the left isomorphism follows from the Gersten resolution and the purity of logarithmic Hodge-Witt sheaves \cite{GrSu}, \cite{Gr}; one also needs the fact that $\Pic(M)$ is killed by $\Tor_{\Z}^\eff(M)$ for $(i,j)=(2,1)$.
See \cite[Chapter I, 5.7.5]{I} for the right isomorphism.
Now assume that $S$ is a {\it supersingular} Enriques surface over $k$ with $\ch(k)=2$,
which satisfies $\Pic^\tau_{S/k} \cong \alpha_2$ \cite[Chapter II, 7.3.1\,(d)]{I}.
Then the unramified cohomology groups are computed as follows:
\par\smallskip
\begin{itemize}
\item[(a)]
We have $H^2(S,W\hspace{-1pt}\sO_S) \cong k$,
 on which the Frobenius operator $F$ is $0$ \cite[Chapter II, 7.3.2]{I}.
Hence $H^2_\et(S,\Z_2)=H^2(S,W\hspace{-1pt}\sO_S)^{F=1}=0$,
and $H^{1,0}_\ur(M)\{2\}=0$.
\item[(b)]
Since $\Pic^\tau_{S/k} \cong \alpha_2$, $H^1_\et(S,W\Omega_{S,\log}^1)_{2\text{-}\Tor}$ is zero,
i.e., $H^{1,1}_\ur(M)\{2\}=0$.
\item[(c)]
Since $H^2(S,W\Omega^1_S) \cong k$ \cite[Chapter II, 7.3.6\,(b)]{I},
we have $H^2_\et(S,W\Omega_{S,\log}^1) \cong \bZ/2\bZ$ or $0$.
Since $\Pic^\tau_{S/k} \cong \alpha_2$, the perfect group scheme
$\ul{H}^0_{\hspace{1pt} \et}(S,\Omega_{S,\log}^1)$ is isomorphic to $\alpha_2$,
and the \'etale part of $\ul{H}^2_{\hspace{1pt}\et}(S,\Omega_{S,\log}^1)$ is zero by the flat duality of Milne \cite[2.7\,(c)]{Mi}, i.e., $H^2_\et(S,\Omega_{S,\log}^1)=0$. Therefore $H^{2,1}_\ur(M)\{2\}=0$.
\item[(d)]
Since $H^1(S,W\Omega_S^2)=0$, 
$H^1_\et(S,W\Omega_{S,\log}^2)$ is zero, i.e., $H^{2,2}_\ur(M)\{2\}=0$.
\end{itemize}
\par\smallskip\noindent
Thus we have $H^{i,j}_\ur(M)\{2\} = 0$ for all $i,j$.
On the other hand, we have $H^2(S,\sO_S) \cong k$.
Since the functor $H^2(-,\sO_{-})$ is normalized, birational, and motivic \cite{CR},
we have $H^2(M,\sO_M) \cong k$ and $M$ is non-zero in $\Chow^\nor_{\Z_2}$.
These facts imply that $b_2$ for $T=S$ is not injective.
\end{enumerate}
\end{remark}

\section{Cohomology of the torsion motive of a surface}
We retain the assumptions and notations introduced in 
Setting \ref{setting:surf} and \ref{setting:def-M}.
We prove a few preliminary lemmas in this section.
To ease the notation, put
\begin{equation}\label{eq:N-B-S}
N_S:=\NS(S)_{\Tor, \Lambda} 
\qquad
B_S:=\Br(S)_\Lambda.
\end{equation}
For a positive integer $m$ invertible in $k$, we denote the Bockstein operator for $m$ by
\begin{equation} \label{eq:bockstein}
Q : H^i_\et(-, \mu_m) \to H^{i+1}_\et(-, \mu_m),
\end{equation}
i.e.,  the connecting map associated to 
the short exact sequence
$0 \to \mu_m \to \mu_{m^2} \to \mu_m \to 0$.

\begin{lemma}\label{lem:etcoh-fincoef}
For any $m \in \Z_{>0}$ invertible in $k$,
we have 
canonical isomorphisms
\begin{equation}\label{eq:etcoh-M-coeff-m}
H^i_\et(M, \mu_m)
\cong
\begin{cases}
0 & (i \not= 1, 2, 3), \\
N_S[m] &(i=1), \\
B_S/m  B_S & (i=3),
\end{cases}
\end{equation}
and an exact sequence
\begin{equation}\label{eq:etcoh}
\xymatrix{
0 \ar[r] &
N_S/m  N_S \ar[r] &
H^2_\et(M, \mu_m) \ar[r] &
B_S[m]  \ar[r] &
0.
}
\end{equation}
If moreover $m  N_S=0$
(so that we have $m  B_S=0$ as well by \eqref{eq:torsion-order-S}),
then we have a commutative diagram
with exact rows
\[
\xymatrix{
0 \ar[r] 
&
H^1_\et(M, \mu_m) \ar[r]^Q \ar[d]^\cong
&
H^2_\et(M, \mu_m) \ar[r]^Q \ar@{=}[d]
&
H^3_\et(M, \mu_m) \ar[r] 
&
0
\\
0 \ar[r] 
&
N_S \ar[r]
&
H^2_\et(M, \mu_m) \ar[r]
&
B_S \ar[r] \ar[u]_\cong
&
0,
}
\]
where 
the vertical isomorphisms are those in \eqref{eq:etcoh-M-coeff-m},
and the lower sequence is obtained from the exact sequence \eqref{eq:etcoh}
with the identifications $N_S/m  N_S=N_S, ~B_S[m]=B_S$.
\end{lemma}
\begin{proof}
The first statement follows from 
Proposition \ref{prop:b1-b2-rho},
Lemma \ref{lem:coh-surf} and
Corollary \ref{cor:coh-S-M} (2),
and the second from the definition of $Q$.
\end{proof}

\begin{lemma}\label{lem:split}
Suppose that $m_0 \in \Z_{>0}$
is invertible in $k$ and $m_0 N_S=0$. 
Put $m:=m_0^2$
and let $Q$ be
the Bockstein operator \eqref{eq:bockstein} for $m$.
Then there exists a subgroup
$\wt{B}_S$ of $H^2_{\et}(M, \mu_{m})$ fitting into 
a commutative diagram with exact row
\begin{equation}\label{eq:et-coh-M}
\xymatrix{
& N_S \ar[d]_{\cong} 
& \wt{B}_S \ar@{^{(}-^{>}}[d] \ar[dr]^{\cong} & &
\\
0 \ar[r] &
H^1_{\et}(M, \mu_{m}) \ar[r]_Q &
H^2_{\et}(M, \mu_{m}) \ar[r]_Q &
H^3_{\et}(M, \mu_{m}) \ar[r] &
0.}
\end{equation}
In particular, we have an isomorphism
\begin{equation}\label{eq:etcoh2-fincoef}
H^2_{\et}(M, \mu_{m}) \cong
QN_S \oplus \wt{B}_S,
\end{equation}
where we identified $N_S=H^1_{\et}(M, \mu_{m})$.
\end{lemma}

\begin{proof}
Put
$H^i_{\et, n}(M):= H^i_{\et}(M, \mu_n)$.
We consider a commutative diagram
with exact rows and columns 
\[
\xymatrix{
& &
H^2_{\et, m_0}(M) \ar[r] \ar[d] &
B_S[m_0] \ar[r] \ar[d]^\cong &
0
\\
0 \ar[r] &
N_S/m N_S \ar[r] \ar[d]^\cong &
H^2_{\et, m}(M) \ar[r] \ar[d] &
B_S[m] \ar[r] & 
0
\\
0 \ar[r] &
N_S/m_0  N_S \ar[r]  &
H^2_{\et, m_0}(M). & 
& 
}
\]
All rows are from \eqref{eq:etcoh}.
The left and right vertical bijections come from 
$m_0N_S = mN_S = 0$ and $B_S[m_0]=B_S[m_0]=B_S$,
which follows from our assumption on $m_0$ and $m$.
We now rewrite it using the latter half of
Lemma \ref{lem:etcoh-fincoef}:
\[
\xymatrix{
& &
H^2_{\et, m_0}(M) \ar[r]^{Q_0} \ar[d]^\iota &
H^3_{\et, m_0}(M) \ar[r] \ar[d]^\cong &
0
\\
0 \ar[r] &
H^1_{\et, m}(M) \ar[r]^Q \ar[d]^\cong &
H^2_{\et, m}(M) \ar[r]^Q \ar[d]^\pi &
H^3_{\et, m}(M) \ar[r] & 
0
\\
0 \ar[r] &
H^1_{\et, m_0}(M) \ar[r]^{Q_0} &
H^2_{\et, m_0}(M), & 
}
\]
where $Q_0$ denotes the Bockstein operator \eqref {eq:bockstein} for $m_0$.
We then obtain the assertion from the middle horizontal exact row
by putting $\wt{B}_S:=\Im(\iota)=\ker(\pi)$.
\end{proof}

\section{Vishik's method}\label{sect:vishik}

In \cite[\S 4]{Vishik},
Vishik obtained an exact sequence
that computes the motivic cohomology with $\Z/5\Z$ coefficients
of the classical Godeaux surface over $\C$.
In this section we apply his method to a general surface 
having a decomposition of the diagonal
over an arbitrary algebraically closed field.
The main result of this section is
Theorem \ref{thm:vishik} below.

We retain the assumptions and notations introduced in 
Setting \ref{setting:surf} and \ref{setting:def-M}.
We also fix the following data:

\begin{setting}\label{setting:tate-twists}
Fix $m_0 \in \Z_{>0}$ that is invertible in $k$ and
divisible by \eqref{eq:torsion-order-S}.
Put $m:=m_0^2$.
We also fix an isomorphism $\Z/m\Z \cong \mu_m$
by which we will identify 
\'etale and Galois cohomology
with different Tate twists.
We write
\[
H^i_{\et}(-):=H^i_\et(-, \Z/m\Z),
\qquad
H^i_{\Gal}(-):=H^i_\Gal(-, \Z/m\Z).
\]
\end{setting}

Using the isomorphism from 
\eqref{eq:ur-coh-12}, \eqref{eq:N-B-S} and \eqref{eq:etcoh-M-coeff-m},
we identify 
\begin{align}
\label{eq:etcoh13-fincoef}
H^1_\ur(S) \cong H^1_\et(M) \cong N_S,
\qquad
H^2_\ur(S) \cong H^3_\et(M) \cong B_S,
\end{align}
which are finite abelian groups dual to each other
by Lemma \ref{lem:coh-surf} (3).

\subsection{Motivic cohomology}

For $X \in \Sm$, $K \in \Fld$,
and $a, b \in \Z$ with $b \ge 0$,
we write
\begin{equation} \label{eq:def-mot-coh}
H^{a, b}_\sM(X_K, \Lambda):= H^a_\Zar(X_K, \Lambda(b)),
\qquad
H^{a, b}_\sM(X_K):= H^a_\Zar(X_K, \Z/m\Z(b)).
\end{equation}
where $\Lambda(b)$ and $\Z/m\Z(b)$ are Voevodsky's motivic complex
\cite[Definition 3.1]{MVW}
with coefficients in $\Lambda$ and $\Z/m\Z$, respectively.
We put
$H^{a, b}_\sM(X_K, \Lambda)
=H^{a, b}_\sM(X_K)=0$ if $b<0$.
We recall the following fundamental facts:
\begin{align}
&
\label{eq:motcoh1}
H^{a, b}_\sM(X_K, \Lambda)=
H^{a, b}_\sM(X_K)=0
\quad \text{if $a>2b$ or $a>b+\dim X$}.
\\
&
\label{eq:motcoh2}
H^{2b, b}_{\sM}(X_K, \Lambda)
\cong \CH^b(X_K)_\Lambda,
\quad
H^{2b, b}_{\sM}(X_K)
\cong \CH^b(X_K)/m \CH^b(X_K),
\\
&
\label{eq:motcoh3}
H^{a, b}_{\sM}(X_K)
\cong H^a_{\et}(X_K)
\quad \text{if $a \le b$}.
\end{align}
The case $a>2b$ of \eqref{eq:motcoh1}
and \eqref{eq:motcoh2} are 
consequences of Voevodsky's comparison theorem
on the motivic cohomology with 
Bloch's higher Chow groups
(see \cite[Corollary 19.2, Theorem 19.3]{MVW}).
The second case of \eqref{eq:motcoh1}
is immediate from the definition
(see \cite[Theorem 3.6]{MVW}).
The former Beilinson-Lichtenbaum conjecture \eqref{eq:motcoh3} 
is proved in \cite[Theorem 6.17]{V-lcoef}
as a consequence of 
Rost-Voevodsky's norm residue isomorphism theorem \cite[Theorem 6.16]{V-lcoef},
based on the previous works of Suslin-Voevodsky \cite{SV}
and Geisser-Levine \cite{GL}.
%

If we fix $a, b$ and $K$ and let $X$ varies,
then $H^{a, b}_\sM(X_K, \Lambda)$ defines a motivic functor.
This follows from
\cite[Propositions 14.16 and 20.1]{MVW},
as
$H^{a, b}_\sM(X_K, \Lambda)$ is the colimit of 
$H^{a, b}_\sM(X \times U, \Lambda)$ 
where $U$ ranges over all smooth schemes over $k$ with function field $K$.
The same is true of $H^{a, b}_\sM(X_K)$.
Therefore the notations and results discussed in the previous paragraph
are extended to motives, 
cf. Remark \ref{rem:extension-F}.

We now state the main result of this section.

\begin{theorem}\label{thm:vishik}
For any $a \in \Z$ and $K \in \Fld$, 
we have an exact sequence
\[
0 
\to H^{a, a-2}_{\sM}(M_K)
\to \bigoplus_{i=1, 2} 
H^{a-i-1}_{\Gal}(K) \otimes H^i_{\ur}(S)
\overset{\Psi}{\to} H^{a-1}_{\ur}(K(S)/K)
\to 0.
\]
Here $\Psi$ is given by 
$\Psi(a \otimes b) = \pr_1^*(a) \cup \pr_2^*(b)$,
where $\pr_i$ denotes the respective projectors 
on $\Spec(K) \times S$.
(The last term is
the unramified cohomology over $K$ and not over $k$.)
\end{theorem}

\subsection{\'Etale cohomology}

\begin{proposition}\label{prop:etcoh-dec}
For any 
$N \in \Chow_\Lambda$ and
$K \in \Fld$, we have an isomorphism
\begin{equation}\label{eq:spec-decomp}
H^*_{\Gal}(K) \otimes H^*_{\et}(N)
\cong H^*_{\et}(N_K).
\end{equation}
\end{proposition}
\begin{proof}
Vishik proved \eqref{eq:spec-decomp} in \cite[Proposition 4.2]{Vishik}
assuming $k=\C$ and $m$ is a prime,
although his proof did not use those assumptions.
For the completeness sake we include a short proof.
We may replace $N$ by $X \in \Sm$.
Consider the spectral sequence
\begin{equation}\label{eq:spec-seq}
E_2^{a, b}=
H^a_{\Gal}(K, H^b_{\et}(X_{\ol{K}}))
\Rightarrow H^{a+b}_{\et}(X_K),
\end{equation}
where $\ol{K}$ is a separable closure of $K$.
By the smooth base change theorem
we have
$H^b_{\et}(X_{\ol{K}}) \cong H^b_{\et}(X)$
on which the absolute Galois group of $K$
acts trivially, and hence
\[
E_2^{a, b}=
H^a_{\Gal}(K, H^b_{\et}(X_{\ol{K}}))
\cong 
H^a_{\Gal}(K) \otimes H^b_{\et}(X).
\]
Observe that
$E_2^{*, *}$ is generated by $H^*_\et(X)$ as a $H^*_\Gal(K)$-module,
and the differential maps
$d_r^{*, *} : E_r^{*, *} \to E_r^{*+r, *-r+1}$
are $H^*_\Gal(K)$-linear.
It follows from the commutative diagram
\[
\xymatrix{
H^j_\et(X_K) \ar[r] 
& E_2^{0, j}=H^0_\Gal(K, H^j_\et(X_{\ol{K}})) \ar@{^{(}-^{>}}[d]
\\
H^j_\et(X) \ar[r]^{\cong} \ar[u]
& H^j_\et(X_{\ol{K}})
}
\]
that the edge maps 
$H^j_\et(X_K) \to E_2^{0, j}$
are surjective for all $j$,
whence $E_2^{0, j}=E_\infty^{0, j}$.
We conclude that \eqref{eq:spec-seq} degenerates at $E_2$-terms
and induces the desired isomorphism.
\end{proof}

\begin{remark}\label{rem:smooth-bc}
The proof shows that
\eqref{eq:spec-decomp} remains valid
when $N$ is replaced by any $X \in \Sm$.
\end{remark}

\begin{corollary}\label{cor:coniveau0}
For any $K \in \Fld$ and $a \in \Z$, 
we have an isomorphism
\begin{align}
\label{eq:coniveau0}
H^{a, a}_{\sM}(M_K) \cong 
&
(H^{a-1}_{\Gal}(K) \otimes N_S) \oplus
(H^{a-2}_{\Gal}(K) \otimes QN_S) 
\\
\notag
&\oplus
(H^{a-2}_{\Gal}(K) \otimes \wt{B}_S) \oplus
(H^{a-3}_{\Gal}(K) \otimes B_S).
\end{align}
\end{corollary}
\begin{proof}
Apply Proposition \ref{prop:etcoh-dec}
to $N=M$ and
use \eqref{eq:etcoh2-fincoef},
\eqref{eq:etcoh13-fincoef} and
\eqref{eq:motcoh3}.
\end{proof}

\subsection{The first coniveau filtration}
The isomorphism
$\Z/m\Z \cong \mu_m$
fixed in Setting \ref{setting:tate-twists}
yields a homomorphism
\[
\tau : H^{a, b}_{\sM}(M_K)
\to H^{a, b+1}_{\sM}(M_K).
\]

\begin{proposition}\label{prop:coniveau1}
For any $K \in \Fld$ and $a \in \Z$, the map
\[
\tau : H^{a, a-1}_{\sM}(M_K)
\to H^{a, a}_{\sM}(M_K)
\cong H^a_{\et}(M_K)
\]
is injective and its image corresponds to the subgroup
\begin{align}
\label{eq:coniveau1}
& 
(H^{a-2}_{\Gal}(K) \otimes QN_S) \oplus
(H^{a-3}_{\Gal}(K) \otimes B_S)
\\
\notag
&\oplus
\ker[\alpha_a :
(H^{a-1}_{\Gal}(K) \otimes N_S) \oplus
(H^{a-2}_{\Gal}(K) \otimes \wt{B}_S)
\to H^a_{\ur}(M_K)]
\end{align}
under the isomorphism \eqref{eq:coniveau0}
(see \eqref{eq:def-unramcoh4} for $H^a_{\ur}(M_K)$).
Here $\alpha_a$ is given by the composition
\[
(H^{a-1}_{\Gal}(K) \otimes N_S) \oplus
(H^{a-2}_{\Gal}(K) \otimes \wt{B}_S)
\overset{\text{\eqref{eq:coniveau0}}}{\hookrightarrow}
H_\sM^{a, a}(M_K) 
\overset{\rho}{\longrightarrow}
H_{\ur}^a(M_K),
\]
where $\rho$ is given by Theorem \ref{thm:TY} (1) below.
\end{proposition}

\begin{remark}
We will show that
$\alpha_a$ is surjective 
in Proposition \ref{prop:coniveau2} below.
\end{remark}

For the proof,
we recall an important result from \cite{TY}:

\begin{theorem}\label{thm:TY}
Let  $X \in \Sm$, $K \in \Fld$
and $a, b \in \Z$ with $b \ge 0$.
\begin{enumerate}
\item 
There exists a long exact sequence
\[
\cdots 
\to H^{a, b-1}_{\sM}(X_K)
\overset{\tau}{\to} H^{a, b}_{\sM}(X_K)
\overset{\rho}{\to} H^{a-b}_{\Zar}(X_K, \sH^b_m)
\to H^{a+1, b-1}_{\sM}(X_K)
\overset{\tau}{\to} \cdots,
\]
where 
$\sH^b_{m}$ is from \eqref{eq:def-unramcoh4}.
\item 
Let 
$E_1^{i, j}=H^{2i+j}_\Zar(X_K, \sH^{-i}_m) \Rightarrow H^{i+j}_\et(X_K)$
be the $\tau$-Bockstein spectral sequence 
constructed in \cite[p. 4478]{TY}
(using the long exact sequence in (1)).
Let
${}^\dagger E_1^{i, j}=\oplus_{x \in (X_K)^{(i)}}
H^{j-i}_\Gal(K(x))
\Rightarrow H^{i+j}_\et(X_K)$
be the coniveau spectral sequence.
Then we have an isomorphism of spectral sequences
$E_r^{i, j} \cong {}^\dagger E_{r+1}^{2i+j, -i}$.
\item 
The composition
\[ \CH^a(X_K)/m \CH^a(X_K)
\cong 
H^{2a, a}_\sM(X_K)
\overset{\tau^a}{\longrightarrow}
H^{2a, 2a}_\sM(X_K) \cong
H^{2a}_\et(X_K)
\]
agrees with the cycle map.
\end{enumerate}
\end{theorem}
\begin{proof}
This is taken from \cite[Lemma 2.1, Theorem 2.4]{TY}.
Here we only recall that
(1) is a consequence of 
\eqref{eq:motcoh3},
(2) is due to Deligne and Paranjape 
(see \cite[p.195, footnote]{BO}, \cite[Corollary 4.4]{P}),
and (3) is a consequence of (2).
\end{proof}

We need a simple lemma.

\begin{lemma}\label{lem:comm-bock}
\begin{enumerate}
\item 
The following diagram is commutative:
\[
\xymatrix{
H^{a, b}_{\sM}(M_K) \ar[r]^\tau &
H^{a, b+1}_{\sM}(M_K) 
\\
H^{a-1, b}_{\sM}(M_K) \ar[r]_\tau \ar[u]^Q&
H^{a-1, b+1}_{\sM}(M_K). \ar[u]_Q
}
\]
\item
We have
$Q(H^a_\Gal(K) \otimes H^b_\et(M)) = H^a_\Gal(K) \otimes Q(H_\et^b(M))$.
\end{enumerate}
\end{lemma}
\begin{proof}
We have $Q(\zeta)=0$ for any $\zeta \in \mu_m$
because 
the $m$-th power map
$H^0_\Gal(k, \mu_{m^2}) \to H^0_\Gal(k, \mu_m)$
is surjective as $k$ is algebraically closed.
Thus (1)  follows from a formal property of
the Bockstein operator 
$Q(x \cup y)=Q(x) \cup y \pm x \cup Q(y)$
by taking $y=\zeta$
(since $\tau = - \cup \zeta$ by definition).
The same formal property reduces (2) to the surjectivity of
$H_\Gal^a(K, \mu_{m^2}^{\otimes a}) \to H_\Gal^a(K, \mu_{m}^{\otimes a})$,
which is a consequence of 
the norm residue isomorphism theorem (see \cite[Theorem 6.16]{V-lcoef}).
\end{proof}

\begin{proof}[Proof of Proposition \ref{prop:coniveau1}]
The injectivity of $\tau$ is 
a part of the Beilinson-Lichtenbaum conjecture 
(proved by Voevodsky in \cite[Theorem 6.17]{V-lcoef}).
Since
$H^{-1}_\Zar(S_K, \sH^a_m)=0$ and
$H^0_\Zar(S_K, \sH^a_m)=H^a_{\ur, m}(S_K)$
by the definition \eqref{eq:def-unramcoh4},
we obtain from Theorem \ref{thm:TY} (1) with $a=b$
an exact sequence 
sitting in the  upper row of a diagram:
\begin{equation}\label{eq:diag-coniv1}
\xymatrix{
0 \ar[r] &
H^{a, a-1}_{\sM}(M_K) \ar[r]^\tau &
H^{a, a}_{\sM}(M_K) \ar[r]^\rho &
H^a_{\ur, m}(M_K)
\\
&
H^{a-1, a-1}_{\sM}(M_K) \ar[r]_\tau^\cong \ar[u]^Q &
H^{a-1, a}_{\sM}(M_K). \ar[u]_Q &
}
\end{equation}
(This reproves the desired injectivity.)
The square in \eqref{eq:diag-coniv1} is 
commutative by Lemma \ref{lem:comm-bock} (1).
The lower horizontal arrow in the diagram
is an isomorphism by \eqref{eq:motcoh3}.
By \eqref{eq:coniveau0}
we find that
$H^{a-1, a-1}_{\sM}(M_K)$ and  $H^{a, a}_{\sM}(M_K)$
are respectively decomposed as
\begin{align*}
&(H^{a-2}_{\Gal}(K) \otimes N_S) \oplus
(H^{a-3}_{\Gal}(K) \otimes QN_S) 
\oplus
(H^{a-3}_{\Gal}(K) \otimes \wt{B}_S) \oplus
(H^{a-4}_{\Gal}(K) \otimes B_S),
\\
&(H^{a-1}_{\Gal}(K) \otimes N_S) \oplus
(H^{a-2}_{\Gal}(K) \otimes QN_S) 
\oplus
(H^{a-2}_{\Gal}(K) \otimes \wt{B}_S) \oplus
(H^{a-3}_{\Gal}(K) \otimes B_S).
\end{align*}
By Lemma \ref{lem:comm-bock} (2) and \eqref{eq:diag-coniv1},
we get
\[
\rho(H^{a-2}_{\Gal}(K) \otimes QN_S)
=
\rho(Q\tau(H^{a-2}_{\Gal}(K) \otimes N_S))
=
\rho(\tau Q(H^{a-2}_{\Gal}(K) \otimes N_S))=0.
\]
Similarly we obtain
$\rho(H^{a-3}_{\Gal}(K) \otimes B_S)=0$
since $B_S=Q\wt{B}_S$.
To conclude \eqref{eq:coniveau1},
it suffices now to note that
$H^a_{\ur, m}(M_K)=H^a_{\ur}(M_K)$
by \eqref{eq:def-unramcoh3}
and use Corollary \ref{cor:coh-S-M} (1).
%
%
%
\end{proof}

\subsection{The second coniveau filtration}

\begin{proposition}\label{prop:coniveau2}
For any $K \in \Fld$ and $a \in \Z$, the map
\[
\tau : H^{a, a-2}_\sM(M_K)
\to H^{a, a-1}_\sM(M_K)
\]
is injective and its image corresponds to the subgroup
\begin{align}
\label{eq:coniveau2}
\ker[\beta_a :
(H^{a-2}_{\Gal}(K) \otimes QN_S) \oplus
(H^{a-3}_{\Gal}(K) \otimes B_S)
\to H^{a-1}_{\ur}(M_K)]
\end{align}
under the isomorphism \eqref{eq:coniveau1}.
Here $\beta_a$ is defined by the commutativity of
\[
\xymatrix{
(H^{a-2}_\Gal(K) \otimes QN_S) \oplus (H^{a-3}_\Gal(K) \otimes B_S)
\ar[rd]^-{\beta_a}
& \\
(H^{a-2}_\Gal(K) \otimes N_S) \oplus (H^{a-3}_\Gal(K) \otimes \wt{B}_S)
\ar[r]_-{\alpha_{a-1}} \ar[u]^Q_\cong
& 
H_\ur^{a-1}(M_K).
}
\]
Moreover, the map $\alpha_a$ in 
\eqref{eq:coniveau1} is surjective.
\end{proposition}
\begin{proof}
Since $H^{a, b}_{\sM}(M_K, \Lambda)$ is
annihilated by $m$ for any $a, b \in \Z$,
a commutative diagram with an exact row
\[
\xymatrix{
H_\sM^{a-1, b}(M_K) \ar@{->>}[r] \ar[rd]_Q &
H_\sM^{a, b}(M_K, \Lambda) \ar[r]^{m=0} \ar@{^{(}-^{>}}[d] &
H_\sM^{a, b}(M_K, \Lambda) \ar@{^{(}-^{>}}[r]  &
H_\sM^{a, b}(M_K)
\\
& H^{a, b}_\sM(M_K) & &
}
\]
shows that 
the complex $(H^{\bullet, b}_{\sM}(M_K), Q)$ is exact.
Consider a diagram 
\[
\xymatrix{
H^{a, a-2}_{\sM}(M_K) \ar[r]^\tau &
H^{a, a-1}_{\sM}(M_K) 
\\
H^{a-1, a-2}_{\sM}(M_K) \ar[r]_\tau 
\ar@{->>}[u]^Q &
H^{a-1, a-1}_{\sM}(M_K), \ar[u]_Q 
}
\]
which is commutative by Lemma \ref{lem:comm-bock} (1).
Since 
$H^{a+1, a-2}_{\sM}(M_K)=0$ by \eqref{eq:motcoh1},
the previous remark shows that
the left vertical map  in the diagram is surjective.
The rest of the proof goes along the same lines as 
Proposition \ref{prop:coniveau1}.
We apply \eqref{eq:coniveau1}
to obtain direct sum decompositions of
$H^{a-1, a-2}_{\sM}(M_K)$ and $H^{a, a-1}_{\sM}(M_K)$
respectively as
\begin{align*}
&
(H^{a-3}_{\Gal}(K) \otimes QN_S) \oplus
(H^{a-4}_{\Gal}(K) \otimes B_S) \oplus
\ker(\alpha_{a-1}),
\\
&
(H^{a-2}_{\Gal}(K) \otimes QN_S) \oplus
(H^{a-3}_{\Gal}(K) \otimes B_S) \oplus
\ker(\alpha_a).
\end{align*}
By Lemma \ref{lem:comm-bock} (2),
the summand
$(H^{a-3}_{\Gal}(K) \otimes QN_S) \oplus
(H^{a-4}_{\Gal}(K) \otimes B_S)$
of $H^{a-1, a-2}_{\sM}(M_K)$
is killed by the left vertical map,
because $Q^2=0$ and $B_S=Q\wt{B}_S$.
On the other hand,
$\tau \circ Q$ maps 
$\ker(\alpha_{a-1})$ injectively into 
the summand
$(H^{a-2}_{\Gal}(K) \otimes QN_S) \oplus
(H^{a-3}_{\Gal}(K) \otimes B_S)$
of
$H^{a, a-1}_{\sM}(M_K)$,
showing the first statement.

In particular, we have shown the injectivity of 
$\tau : H^{a+1, a-1}_\sM(X_K) \to H^{a+1, a}_\sM(X_K)$.
Thus the exact sequence from Theorem \ref{thm:TY} (1) applied with $a=b$
shows that $\rho : H^{a, a}_\sM(X_K) \to H_\ur^a(M_K)$ is surjective.
The same exact sequence together with Proposition \ref{prop:coniveau1} shows that
$\rho((H^{a-2}_{\Gal}(K) \otimes QN_S) \oplus (H^{a-3}_{\Gal}(K) \otimes B_S))=0$.
This completes the proof of the last statement.
%
%
%
%
%
%
%
%
%
%
%
%
\end{proof}

\begin{proof}[Proof of Theorem \ref{thm:vishik}]
As the unramified cohomology is 
normalized, birational, and motivic
(Proposition \ref{prop:unram-coh}),
we have
$H^i_{\ur}(S)=H^i_{\ur}(M)$
and
$H^i_{\ur}(K(S)/K)=H^i_{\ur}(M_K)$.
Now
Propositions \ref{prop:coniveau1} and \ref{prop:coniveau2}
complete the proof.
\end{proof}

\section{Main exact sequence}\label{sect:proof}

We keep the assumptions in 
Setting \ref{setting:surf}, \ref{setting:def-M} and \ref{setting:tate-twists}.

\subsection{Main exact sequence}

The following is the main technical 
result of this paper.

\begin{theorem}\label{thm:tech-main}
Suppose that $S \in \SmProj$ 
admits a decomposition of the diagonal
(see Definition \ref{def:dec-diag})
and $\dim S=2$.
Then we have an exact sequence
for any $K \in \Fld$
\[
0 \to \CH_0(S_K)_{\Tor, \Lambda} \to 
\bigoplus_{i=1, 2} \Hom(H_\ur^i(S), H_\ur^i(K/k)) \to 
H^3_\ur(K(S)/k) \to 0.
\]
(Unlike Theorem \ref{thm:vishik},
the last term is
the unramified cohomology over $k$ and not over $K$.)
\end{theorem}

The proof of Theorem \ref{thm:tech-main}
will be complete in \S \ref{sect:end} below.

\begin{remark}
In the situation of Theorem \ref{thm:tech-main}, 
we have a canonical isomorphism
\begin{equation}
\CH_0(S_K)_{\Tor, \Lambda} \cong 
\Coker(\CH_0(S)_\Lambda \to \CH_0(S_K)_\Lambda),
\end{equation}
and this group is annihilated by 
the integer \eqref{eq:torsion-order-S}.
To see this,
it suffices to note that
the degree map $\CH_0(S_K) \to \Z$ is split surjective
(as $k$ is algebraically closed),
and use Lemma \ref{lem:coh-surf} (4).
As a special case where $K=k(T)$ 
for $T \in \SmProj$,
we also have
(see \eqref{eq:CH0-S})
\begin{equation}\label{eq:CH0S-CHtor}
\CH_0(S_{k(T)})_{\Tor, \Lambda} \cong 
\Chow^\nor_\Lambda(T, S).
\end{equation}
\end{remark}

\subsection{Auxiliary lemmas}

\begin{lemma}\label{lem:freeness}
Let $E$ be a field such that
$m$ is invertible in $E$ 
and $\mu_{m^\infty} \subset E$.
Then $H^j_\Gal(E)$
is a free $\Z/m\Z$-module for any $j \in \Z$.
\end{lemma}
\begin{proof}
We may assume $m=\ell^e$ 
for a prime number $\ell \not= p$ and $e \in \Z_{>0}$.
Recall that a module over an Artin local ring
is free if and only if it is flat
(see, e.g. \cite[Proposition 2.1.4]{A}).
By the norm residue isomorphism theorem (see \cite[Theorem 6.16]{V-lcoef}), 
$K_{j-1}^M(E) \otimes \mu_{\ell^\infty}$
surjects onto 
$K_j^M(E)_{\Tor} \otimes \Z_{(\ell)}$,
hence
$K_j^M(E)_{\Tor}$ is divisible by $\ell$.
It follows that $K_j^M(E)$
is the direct sum of an $\ell$-divisible group
and a flat $\Z_{(\ell)}$-module.
Thus
$K_j^M(E) \otimes \Z/m\Z \cong H^j_\Gal(E)$
is a flat $\Z/m\Z$-module.
\end{proof}


By the Poincar\'e duallty,
we have a perfect paring of finite abelian groups
for any $i \in \Z$
\[
\langle -, -\rangle : 
H^{4-i}_\et(S) \times H^i_\et(S)
\to \Z/m\Z.
\]
For $i=1, 2$, we define the homomorphisms
\begin{align}
\label{eq:def-q-pi}
Q_i' : H^{3-i}_\ur(S) \to H^{4-i}_\et(S),
\qquad
\pi_i : H^i_\et(S) \to H^i_\ur(S)
\end{align}
as follows.
For $i=1$, they are given by \eqref{eq:etcoh13-fincoef}.
For $i=2$, $Q_2'$ and $\pi_2$ are the compositions
\[
H^1_\ur(S) \cong H^1_\et(S) \overset{Q}{\to} H^2_\et(S),
\qquad
H^2_\et(S) \overset{Q}{\to} H^3_\et(S) \cong H^2_\ur(S),
\]
where $Q$ are the Bockstein operator \eqref{eq:bockstein}.
(Hence $Q_1'$ and $\pi_1$ are bijective,
and we have a split short exact sequence
$0 \to H_\ur^1(S) \overset{Q_2'}{\to} H^2_\et(S) \overset{\pi_2}{\to}
H_\ur^2(S) \to 0$.)

\begin{lemma}\label{lem:pi-q}
We have a perfect paring of finite abelian groups
for $i=1, 2$
\[
\langle -, -\rangle :
H^{3-i}_\ur(S) \times H^i_\ur(S) \to \Z/m\Z
\]
characterized by the formula
\begin{equation}\label{eq:pairing}
\langle Q_i'(a), b \rangle = \langle a, \pi_i(b) \rangle
\qquad (a \in H^{3-i}_\ur(S), \, b \in H^i_\et(S)).
\end{equation}
\end{lemma}
\begin{proof}
For $i=1$, \eqref{eq:pairing} is 
nothing other than the paring in Lemma \ref{lem:coh-surf} (3),
whence the result.
Assume now $i=2$.
We claim that
$Q_2'(H_\ur^1(S))$ is the exact annihilator of itself
with respect to $\langle -, - \rangle$.
For this, we first note that
$\langle Q_2'(H_\ur^1(S)), Q_2'(H_\ur^1(S)) \rangle=0$ because 
\[
Q(a) \cup Q(b)
=Q(a) \cup Q(b) - a \cup Q^2(b)
= Q(a \cup Q(b)) = 0
\]
for $a, b \in H_\et^1(S)$.
Here the first (resp. third) equality holds
because $Q^2=0$
(resp. 
$Q : H^3_\et(S) \to H^4_\et(S)$ is the zero map,
as $H^4_\et(S, \Z/m\Z) \to H^4_\et(S, \Z/m^2\Z)$ is injective).
We then use the fact 
$|Q_2'(H_\ur^1(S))|=|H^1_\ur(S)|=|H^2_\ur(S)|=|H^2_\et(S)/Q_2'(H_\ur^1(S))|$
to conclude the claim.
It follows that $\langle -, - \rangle$ induces the perfect paring in the statement
characterized by \eqref{eq:pairing}.
\end{proof}

\begin{lemma}\label{lem:tensor-hom}
Let $E$ be a field satisfying the assumption of
Lemma \ref{lem:freeness}.
Then for $i=1, 2$ and for any $j \in \Z$, we have  isomorphisms
\begin{align*}
&H^j_\Gal(E) \otimes H^{4-i}_\et(S) \cong
\Hom(H^i_\et(S), H^j_\Gal(E)),
\\
&H^j_\Gal(E) \otimes H^{3-i}_\ur(S) \cong
\Hom(H^i_\ur(S), H^j_\Gal(E)).
\end{align*}
\end{lemma}
\begin{proof}
This follows from
Lemmas \ref{lem:freeness}, \ref{lem:pi-q} and \ref{lem:homalg0} (2).
\end{proof}

\begin{lemma}\label{lem:inj}
The canonical map
$H^2_\et(\Spec(E \otimes_k k(S)))
\to
H^2_\Gal(E(S))$
is injective
for any $E \in \Fld$. 
\end{lemma}
\begin{proof}
We consider a commutative diagram with exact row:
\[
\xymatrix{
& & 
H^2_\et(\Spec(E \otimes_k k(S))) \ar[r] &
H^2_\Gal(E(S)) &
\\
0 \ar[r] &
\Pic(U_E)/m \Pic(U_E) \ar[r]  &
H^2_\et(U_E) \ar[r]_{\gamma_U} \ar[u] &
H^2_{\ur, m}(U_E) \ar[r] \ar@{_{(}-{>}}[u]_{\iota_U}&
0,
}
\]
where $U$ is an open dense subscheme of $S$.
Since the map in question is obtained as the 
colimit of $\iota_U \circ \gamma_U$ 
as $U$ ranges over such schemes,
it suffices to show the vanishing of
the lower left group 
for sufficiently small $U \subset S$.
For this, we take a (possibly reducible) curve 
$C \subset S$ whose components generate $\NS(S)$.
Then we find $\Pic(U_E)=0$
as soon as $U \subset S \setminus C$,
because we have $\Pic(S_E)=\NS(S)$ 
by Lemma \ref{lem:coh-surf} (4).
We are done.
\end{proof}

\begin{lemma}\label{lem:inj2}
For any $E \in \Fld$,
the map
\begin{equation}\label{eq:phi}
\bigoplus_{i=1, 2}
H^{i-1}_\Gal(E) \otimes H^{3-i}_\ur(S)
\to 
H^2_\Gal(E(S)),
\quad
a \otimes b \mapsto \pr_1^*(a) \cup \pr_2^*(b)
\end{equation}
is injective,
where $\pr_i$ denotes the respective projectors 
on $\Spec(E) \times S$.
\end{lemma}
\begin{proof}
We decompose \eqref{eq:phi} as follows:
\begin{align*}
\bigoplus_{i=1, 2}
H^{i-1}_\Gal(E) \otimes H^{3-i}_\ur(S)
&\hookrightarrow
\bigoplus_{i=1, 2}
H^{i-1}_\Gal(E) \otimes H^{3-i}_\Gal(k(S))
\\
&\hookrightarrow
H^2_\et(\Spec(E \otimes_k k(S)))
\hookrightarrow
H^2_\Gal(E(S)).
\end{align*}
The injectivity of the first map
follows from Lemma \ref{lem:freeness},
since $H^i_\ur(S)=H^i_\ur(k(S)/k)$ 
is a subgroup of $H^i_\Gal(k(S))$ by definition
(see \eqref{eq:def-unramcoh1}, \eqref{eq:unramcoh-K-S}).
The second (resp. third) map is also injective
by Remark \ref{rem:smooth-bc} (resp. Lemma \ref{lem:inj}).
\end{proof}

\subsection{End of the proof}\label{sect:end}
We consider a commutative diagram
\[
\xymatrix{
0 \ar[r] &
\CH_0(S_K)_{\Tor, \Lambda} \ar[r] &
\underset{i=1, 2}{\bigoplus}
H^i_{\Gal}(K) \otimes H^{3-i}_{\ur}(S) 
\ar[r]^-\Psi \ar[d]^{\partial_1} &
H^3_{\ur}(K(S)/K) \ar[r] \ar[d]^{\partial_2} &
0
\\
& &
\underset{i=1, 2}{\bigoplus}
\underset{v}{\bigoplus}
H^{i-1}_{\Gal}(F_v) \otimes H^{3-i}_{\ur}(S) \ar[r]_-\psi &
\underset{w}{\bigoplus} H^2_\Gal(F_w).  &
}
\]
The upper row is an exact sequence
obtained by setting $a=4$ 
and replacing $i$ with $3-i$
in Theorem \ref{thm:vishik}.
In the lower row,
$v$ (resp. $w$)
ranges over all discrete valuations of $K$
(resp. $K(S)$) that are trivial on $k$,
and $F_v$ (resp. $F_w$)
denotes the residue field.
For each $v$, let $w(v)$ be 
an extension of $v$ to $K(S)$.
Then the $(v, w(v))$-component of 
$\psi$ is given by \eqref{eq:phi} for $E=F_v$,
and the other components are zero.
The two vertical maps are the residue maps
recalled in \S \ref{sect:unram-coh}.

Lemma \ref{lem:inj2} shows that $\psi$ is injective.
By Lemma \ref{lem:tensor-hom} we have
isomorphisms
\begin{align*}
&H^i_{\Gal}(K) \otimes H^{3-i}_{\ur}(S)
\cong
\Hom(H^i_{\ur}(S), H^i_{\Gal}(K)),
\\
&H^{i-1}_{\Gal}(F_v) \otimes H^{3-i}_{\ur}(S)
\cong
\Hom(H^i_{\ur}(S), H^{i-1}_{\Gal}(F_v)).
\end{align*}
By \eqref{eq:def-unramcoh1} and
the left exactness of 
$\Hom(H^i_{\ur}(S), -)$,
we obtain
\begin{align*}
\ker({\partial_1})=
\underset{i=1, 2}{\bigoplus}
\Hom(H^i_\ur(S), H^i_\ur(K/k)).
\end{align*}
On the other hand, since
$H^3_\ur(K(S)/k) \subset H^3_\ur(K(S)/K) \subset H^3_\Gal(K(S))$
we have
\begin{align*}
\ker({\partial_2})
&= H^3_\ur(K(S)/K) \cap \ker(H^3_\Gal(K(S)) \to \bigoplus_w H^2_\Gal(F_w))
=H^3_\ur(K(S)/k).
\end{align*}
Now a diagram chase completes
the proof of
Theorem \ref{thm:tech-main}.
\qed

\begin{remark}
It is not always the case that
$H^i_\ur(K/k) \otimes H^{3-i}_{\ur}(S)
\cong
\Hom(H^i_{\ur}(S), H^i_\ur(K/k))$.
\end{remark}

\section{Main results}

In this section,
we suppose $k$ is algebraically closed
and $\Lambda=\Z[1/p]$.

\subsection{An exact sequence}

\begin{theorem}\label{thm:main2-full}
Let $S, T \in \SmProj$.
Suppose that $S$ 
admits a decomposition of the diagonal
and $\dim S=2$.
Then we have an exact sequence
\begin{equation}\label{eq:Vishik-ex-seq-full}
0 \to \CH_0(S_{k(T)})_{\Tor, \Lambda} 
\overset{\Phi}{\to} 
\bigoplus_{i=1, 2} \Hom(H_\ur^i(S), H_\ur^i(T)) 
\to H^3_\ur(S \times T) \to 0.
\end{equation}
\end{theorem}
\begin{proof}
Apply Theorem \ref{thm:tech-main} to $K=k(T)$
and use \eqref{eq:unramcoh-K-S}.
Note that the injectivity of $\Phi$
follows also from Proposition \ref{prop:GG-V2}
together with \eqref{eq:CH0S-CHtor}.
\end{proof}

\begin{remark}\label{ex:Bruno2}
Using Lemma \ref{lem:homalg},
we may rewrite \eqref{eq:Vishik-ex-seq-full}
as follows:
\begin{equation}
0 \to \CH_0(S_{k(T)})_{\Tor, \Lambda} 
\to
\bigoplus_{i=1, 2} \Tor(H_\ur^{3-i}(S), H_\ur^i(T)) 
\to H^3_\ur(S \times T) \to 0.
\end{equation}
This, together with \eqref{eq:etcoh13-fincoef},
recovers Kahn's exact sequence
\cite[Corollary 6.4]{K1}
as a special case $T=S$.
It also recovers
\cite[Corollary 6.5]{K1}
as the case $\dim T=1$.
The general case should 
be compared with \cite[Theorem 6.3]{K1},
where the map
\[
\CH_0(S_{k(T)})_{\Tor, \Lambda} 
\to
\bigoplus_{i=1, 2} 
\prod_{\ell \not= p}
\Tor(H_\et^{3-i}(S, \Z_\ell), H_\et^i(T, \Z_\ell)) 
\]
is studied.
\end{remark}

\subsection{Faithful property of unramified cohomology}

\begin{theorem}\label{thm:main1-full}
Let $S, T \in \SmProj$.
Suppose that $S$ 
admits a decomposition of the diagonal
and $\dim S=2$.
Let $f : T \to S$ be a morphism in $\Chow^\nor_\Lambda$.
Then the following are equivalent:
\begin{enumerate}
\item We have $f=0$ in $\Chow^\nor_\Lambda(T, S)$.
\item The map
$F(f) : F(S) \to F(T)$ vanishes
for any normalized, birational, and motivic functor 
$F : \SmProj^\op \to \lMod$.
\item The map
$H^i_\ur(f) : H^i_\ur(S) \to H^i_\ur(T)$ 
vanishes for $i=1, 2$.
\end{enumerate}
\end{theorem}
\begin{proof}
The implications
(1) $\Rightarrow$ (2) $\Rightarrow$ (3) are obvious,
and (3) $\Rightarrow$ (1)
follows from Theorem \ref{thm:main2-full}
and Lemma \ref{lem:compatible} below.
\end{proof}

\begin{lemma}\label{lem:compatible}
Under the identification
$\CH_0(S_{k(T)})_{\Tor, \Lambda}
= \Chow^\nor_\Lambda(T, S)$
from \eqref{eq:CH0S-CHtor},
the map $\Phi$ in \eqref{eq:Vishik-ex-seq-full}
is induced by  the functors $H_\ur^i$ for $i=1, 2$.
\end{lemma}
\begin{proof}
Put $K:=k(T)$.
We use a cartesian diagram
\[
\xymatrix{
S_K \ar[r]^-{\pr_2} \ar[d]_-{\pr_1} &
\Spec K \ar[d]^-{s_2} 
\\
S \ar[r]^-{s_1} &
\Spec k,
}
\]
where $\pr_i$ are the projections 
and $s_i$ are the structure maps.
We first show,
by a standard argument,
the commutativity of the diagram
\begin{equation}\label{eq:comm-ku-pd}
\xymatrix{
H^4_\et(S_K) \ar[r]^-\co  &
\bigoplus_i \Hom(H^i_\et(S), H^i_\Gal(K))
\\
&
\bigoplus_i H_\et^{4-i}(S) \otimes H_\Gal^i(K),
\ar[lu]^-\ku_(0.4)\cong \ar[u]_-\pd^-\cong
}
\end{equation}
where $\co$ is the correspondence action
(that is,
$\co(\xi)(a)=\pr_{2*}(\pr_1^*(a) \cup \xi)$),
$\ku$ is the K\"unneth isomorphism,
and
$\pd$ is the isomorphism from Lemma \ref{lem:tensor-hom}.
We take
$a \in H^i_\et(S)$, 
$b \in H^{4-i}_\et(S)$ and
$x \in H^i_\Gal(K)$, 
and compute
\begin{align*}
(\co \circ \ku)(b \otimes x)(a)
&= \pr_{2*}(\pr_1^*(a) \cup \pr_1^*(b) \cup \pr_2^*(x))
\\
&= \pr_{2*}(\pr_1^*(a \cup b) \cup \pr_2^*(x))
\overset{(1)}{=} \pr_{2*}(\pr_1^*(a \cup b)) \cup x
\\
&
\overset{(2)}{=} s_2^* s_{1*}(a \cup b) \cup x
=\pd(b \otimes x)(a).
\end{align*}
Here we have used the projection formula for \'etale cohomology
and the base change property in \cite[Expos\'e XVIII, Th\'eor\`eme 2.9]{SGA4} at $(1)$ and $(2)$, respectively.
We have shown the commutativity of \eqref{eq:comm-ku-pd}.

We now consider the following diagram
\[
\xymatrix{
\CH_0(S_K)_{\Tor, \Lambda} \ar[r]^-\cyc \ar[rrd]_-{(**)}&
H^4_\et(S_K) \ar[r]^-{(*)} &
\bigoplus_i \Hom(H^i_\et(S), H^i_\Gal(K))
\\
& & \bigoplus_i \Hom(H^i_\ur(S), H^i_\ur(T)). \ar@{_{(}-{>}}[u]_{\Pi}
}
\]
Here $\cyc$ is the cycle map,
and $\Pi$ is the direct sum of the compositions
\[
\Hom(H^i_\ur(S), H^i_\ur(T)) 
\hookrightarrow
\Hom(H^i_\ur(S), H^i_\Gal(K)) 
\overset{\pi_i^*}{\hookrightarrow}
\Hom(H^i_\et(S), H^i_\Gal(K)),
\]
where $\pi_i^*$ is induced by $\pi_i$ in \eqref{eq:def-q-pi}
(which is split surjective).
If we set $\pd \circ \ku^{-1}$ at $(*)$ and  $\Phi$ at $(**)$,
then the diagram commutes by Theorem \ref{thm:TY} (3) and Lemma \ref{lem:pi-q}.
On the other hand,
if we set $\co$ at $(*)$ and the induced map by $H^*_\ur$ at $(**)$,
then the diagram commutes by definition.
Hence the assertion follows from the commutativity of \eqref{eq:comm-ku-pd}.
%
%
%
%
%
%
%
%
%
%
%
\end{proof}

\begin{example}\label{ex:Beauville}
Let $S$ be an Enriques surface over $\C$
(so that $S$ admits a decomposition of the diagonal 
by \cite{BKL} and Remark \ref{rem:div-nor-eff} (1)).
Let $f : T \to S$ be its universal cover
so that $\deg(f)=2$ and $T$ is a $K3$ surface.
In \cite[Corollary 5.7]{B},
Beauville showed that
$H^2_\ur(f)$ vanishes if and only if 
there exists $L \in \Pic(T)$ such that
$\sigma(L)=L^{-1}$ and
$c_1(L)^2 \equiv 2 \bmod 4$,
where $\sigma \in \Gal(f)$ is the non-trivial element.
Moreover, it is shown that
all the $S$ satisfying those conditions
form an infinite countable union of hypersurfaces 
in the moduli space of Enriques surfaces \cite[Corollary 6.5]{B}.
Explicit examples of $S$ satisfying those conditions
can be found in \cite{GS, HS}.
As $H^1_\ur(f)=0$ by definition,
Theorem \ref{thm:main1-full} shows that
this condition implies $F(f)=0$ 
for any normalized, birational, and motivic functor $F$.
\end{example}

\begin{example}\label{ex:Bruno}
Let us apply Theorem \ref{thm:main1-full} 
to $T=S$ and $f=m \cdot \id_S$ with $m \in \Z_{>0}$.
The minimal $m$ which satisfies
the condition (3)
is nothing other than
the torsion order $\Tor_\Lambda^\nor(S)$
in the sense of Definition \ref{def:tor-ord}.
Thus Theorem \ref{thm:main1-full} 
(together with \eqref{eq:etcoh13-fincoef})
recovers 
a main result of \cite[Corollary 6.4 (b)]{K1},
which says $\Tor_\Lambda^\nor(S)
=\exp(\NS(S)_{\Lambda, \Tor})$.
\end{example}

Theorem \ref{thm:main1-full} suggests
the following problem.

\begin{problem}\label{prob:faithful}
Is the functor $H^*_\ur$,
viewed as a functor
from the full subcategory of 
torsion objects in $\Chow^\nor_\Lambda$
to $\lMod$,
faithful?
(Compare \cite[Question 3.5]{K1}.)
\end{problem}

\subsection{Explicit computation of the Chow group and unramified cohomology}

\begin{theorem}\label{thm:main3-full} 
Suppose the characteristic of $k$ is zero.
Let $S \in \SmProj$ be a surface admitting
a decomposition of the diagonal.
If $H^1_\ur(S)$ is a cyclic group of prime order $\ell$,
then so are 
$\CH_0(S_{k(S)})_{\Tor, \Lambda}$ 
and $H^3_\ur(S \times S)$.
\end{theorem}
\begin{proof}
Let $M \in \Chow^\eff_\Lambda$ 
be the Chow motive constructed in Proposition \ref{prop:GG-V}.
Since 
$\CH_0(S_{k(S)})_{\Tor, \Lambda} 
= \Chow^\nor_\Lambda(S, S) = \Chow^\nor_\Lambda(M, M)$,
Proposition \ref{prop:GG-V2} (1) 
and \eqref{eq:Vishik-ex-seq-full} 
yields an exact sequence
\begin{equation}\label{eq:Vishik-ex-seq-full2}
0 \to \Chow_\Lambda^\eff(M, M) 
\overset{\Phi}{\to} 
\bigoplus_{i=1, 2} \Hom(H_\ur^i(S), H_\ur^i(S)) 
\to H^3_\ur(S \times S) \to 0.
\end{equation}
We know $\id_M \in \Chow^\eff_\Lambda(S, S)$ has order $\ell$
by Proposition \ref{prop:GG-V2} (2).
Thus it suffices to show $\Phi$ is not surjective.
If it were surjective,
then by \eqref{eq:Vishik-ex-seq-full2}
there should be a projector $\pi : M \to M$ in $\Chow^\eff_\Lambda$
such that $N:=\Im(\pi) \subset M$ satisfies 
$\Pic(N)=0$ and $\Br(N) \cong \Z/\ell\Z$,
but this would contradict the following result of Vishik.
\end{proof}

\begin{theorem}[Vishik]
Suppose that $k$ is of characteristic zero,
and let $N \in \Chow^\eff_\Lambda$ be a non-trivial 
direct summand of a motive of a surface
such that $\ell \cdot \id_N=0$ for some prime number $\ell$.
Then we have $\Pic(N) \not= 0$.
\end{theorem}
\begin{proof}
See \cite[Corollary 4.22]{Vishik}.
\end{proof}

\begin{remark}
The assumption on the characteristic is used
only to invoke Vishik's result.
It is likely to hold in any characteristic,
as long as $\ell$ is invertible in $k$.
\end{remark}

\begin{corollary}\label{cor:integHC}
In Theorem \ref{thm:main3-full}, suppose further that $k=\C$.
Then we have
\[
\Coker(\CH^2(S \times S) \to 
H^4(S \times S(\C), \Z(2)) \cap H^{2, 2}(S \times S))
\cong \Z/\ell \Z.
\]
In particular, $S \times S$ violates 
the integral Hodge conjecture in codimension two.
\end{corollary}
\begin{proof}
Set $X:=S \times S$.
We claim that $\CH_0(X) \cong \Z$.
For this, it suffices to show that
$\ker(\CH_0(X) \to \Z)$ is torsion by Roitman's theorem,
but Proposition \ref{prop:GG-V} implies that
\[
\ker(\CH_0(X) \to \Z) \cong \Chow^\eff_\Lambda(\Lambda(0), M \otimes M),
\]
which is obviously killed by $\ell$.
Now the corollary is a consequence of
Theorem \ref{thm:main3-full}
and the following result 
of Colliot-Th\'el\`ene and Voisin \cite{CTV}.
\end{proof}

\begin{theorem}[Colliot-Th\'el\`ene, Voisin]\label{thm:CTV}
Suppose $k=\C$ and let $X \in \SmProj$.
Assume that there exist $Y \in \SmProj$ and a morphism $f : Y \to X$ 
such that $\dim Y=2$ and 
$f_* : \CH_0(Y) \to \CH_0(X)$ is surjective.
Then we have an isomorphism of finite abelian groups
\[
H^3_\ur(X)
\cong 
\Coker(\CH^2(X) \to 
H^4(X(\C), \Z(2)) \cap H^{2, 2}(X)).
\]
\end{theorem}
\begin{proof}
See \cite[Th\'eor\`eme 3.9]{CTV}.
\end{proof}

\begin{example}\label{ex:chow-computation}
\begin{enumerate}
\item 
By applying
Theorem \ref{thm:main3-full}
to an Enriques surface $S$,
we find that
$\CH_0(S_{k(S)})_\Tor$ is of order two.
This answers a question raised by
Kahn \cite[p. 840, footnote]{K1}
(in case of characteristic zero).
\item 
Similarly, 
we may apply Theorem \ref{thm:main3-full}
to a Godeaux surface $S$ over $\C,$
as long as Bloch's conjecture holds for $S$
(see Remark \ref{rem:div-nor-eff}).
This is previously known for the classical Godeaux surface
by Vishik (see a remark after Proposition 4.6 in \cite{Vishik}).
Other Godeaux surfaces
for which Bloch's conjecture is verified 
can be found in \cite{GP, Voisin}.
\end{enumerate}
\end{example}

\begin{problem}
Does the equality
\[ |\CH_0(S_{k(S)})_{\Tor}|=|H^3_\ur(S \times S)| \]
remain valid when $H^1_\ur(S)  \cong \NS(S)_{\Tor, \Lambda}$ 
is not cyclic of prime order,
e.g. for a Beauville surface (see \cite{GS2})
or for a Burniat surface (see \cite{AO}) over $\C$?
Note that
Bloch's conjecture is known for such surfaces,
and
we have 
$H^1_\ur(S) \cong \Z/5\Z \times \Z/5\Z$
or 
$H^1_\ur(S) \cong \Z/2\Z \times \Z/2\Z \times \Z/2\Z$,
respectively.
\end{problem}

\section{Appendix : elementary homological algebra}

In this section we prove some elementary lemmas
that have been used in the body of this paper.

\begin{lemma}\label{lem:homalg0}
\begin{enumerate}
\item 
Let $A, B$ be abelian groups.
Suppose 
that $A$ is finitely generated
and that $B$ is a free $\Z$-module.
Then the canonical map
\[
\Hom(A, \Q/\Z) \otimes B \to \Hom(A, B \otimes \Q/\Z),
\qquad
\chi \otimes b \mapsto [a \mapsto b \otimes \chi(a)]
\]
is an isomorphism.
\item 
Let $m \in \Z_{>0}$
and let $A, B$ be $\Z/m\Z$-modules.
Suppose 
that $A$ is finite 
and that $B$ is a free $\Z/m\Z$-module.
Then the canonical map
\[
\Hom(A, \Z/m\Z) \otimes B \to \Hom(A, B),
\qquad
\chi \otimes b \mapsto [a \mapsto \chi(a)b]
\]
is an isomorphism.
\end{enumerate}
\end{lemma}
\begin{proof}
(1) Write $B=\Z^{\oplus I}$ with some set $I$.
Since tensor product commutes with arbitrary sums,
we can identify
$-\otimes B = (-)^{\oplus I}$.
To conclude, it suffices to note that
$\Hom(A, -)$ commutes with arbitrary sums
because  $A$ is  finitely generated.
The proof of (2) is identical.
\end{proof}

\begin{lemma}\label{lem:homalg2}
Let $A, B$ be abelian groups.
Suppose that $A$ is finite
and that $B$ is  a free $\Z$-module.
Then we have canonical isomorphisms
\[
\Hom(A, \Q/\Z) \otimes B \cong 
\Hom(A, B \otimes \Q/\Z)
\cong \Ext(A, B).
\]
\end{lemma}
\begin{proof}
The first isomorphism is from
Lemma \ref{lem:homalg0}.
The second is seen by an exact sequence
$0 \to B \to B \otimes \Q \to B \otimes \Q/\Z \to 0$,
together with
$\Hom(A, B \otimes \Q)=\Ext(A, B \otimes \Q)=0$
as $A$ is finite and $B \otimes \Q$ is injective.
\end{proof}

\begin{lemma}\label{lem:homalg}
Let $A, B$ be abelian groups
with $A$ finite.
Then we have canonical isomorphisms
\[
\Tor(\Hom(A, \Q/\Z), B) \cong \Hom(A, B),
\qquad
\Hom(A, \Q/\Z) \otimes B \cong \Ext(A, B).
\]
\end{lemma}
\begin{proof}
Set $(-)^\vee := \Hom(-, \Q/\Z)$.
We take an exact sequence
$0 \to B_1 \to B_0 \to B \to 0$
with free $\Z$-modules $B_i$.
Applying the
two functors $A^\vee \otimes -$ and $\Hom(A, -)$,
we obtain a commutative diagram with exact rows
\[
\xymatrix{
0 \ar[r] &
\Tor(A^\vee, B) \ar[r]  &
A^\vee \otimes B_1 \ar[r] \ar[d]^\cong &
A^\vee \otimes B_0 \ar[r] \ar[d]^\cong &
A^\vee \otimes B \ar[r]  &
0
\\
0 \ar[r] &
\Hom(A, B) \ar[r] &
\Ext(A, B_1) \ar[r] &
\Ext(A, B_0) \ar[r] &
\Ext(A, B) \ar[r] &
0,
}
\]
where two middle vertical isomorphisms
are from Lemma \ref{lem:homalg2}.
The lemma follows.
\end{proof}

\section{Appendix : $\P^1$-invariance and birational motives}\label{app:bruno}
The aim of this appendix is to prove Proposition \ref{prop:P1inv-bir} below.
We freely use the basic notion from \cite{MVW}.
Let $F$ be a Nisnevich sheaf with transfers over our base field $k$.
For $\epsilon = 0, 1$,
we denote by $i_\epsilon : \Spec k \to \A^1$ 
the corresponding closed immersions and define
\begin{align*}
&h_0(F):=\Coker(i_0^*-i_1^* : \Hom_\PST(\Z_\tr(\A^1), F) \to F),
\\
&\ol{h}_0(F):=\Coker(i_0^*-i_1^* : \Hom_\PST(\Z_\tr(\P^1), F) \to F)
\end{align*}
as presheaf cokernels.
For an abelian group $A$, we write
$F \otimes A$ for a presheaf with transfers
given by $U \mapsto F(U) \otimes_\Z A$.
Note that the canonical map 
\begin{equation}\label{eq:iso-nis}
(F \otimes A)_\Nis \to (F_\Nis \otimes A)_\Nis 
\end{equation}
is an isomorphism
(being a map of sheaves that induces isomorphisms on stalks).
The following proposition is communicated to us by Bruno Kahn.

\begin{proposition}[B. Kahn]\label{prop:P1inv-bir}
Let $G$ be a $\P^1$-invariant Nisnevich sheaf with transfers.
For any $X \in \Sm$ connected and for any $Y \in \SmProj$,
there is a homomorphism $(*)$ fitting in a commutative diagram
\[
\xymatrix{
\Cor(X, Y) \ar[rr] \ar[d]
& & \Hom_{\Ab}(G(Y), G(X))
\\
\Cor(\Spec k(X), Y) \ar@{=}[r] 
&  Z_0(Y_{k(X)}) \ar[r]
& \CH_0(Y_{k(X)}). \ar[u]_{(*)}
}
\]
%
In particular, $G$ is birational and motivic 
in the sense of Definition \ref{def:bir-mot-inv}
(with $\Lambda=\Z$).
\end{proposition}
\begin{proof}
We consider the following diagram:
\[ 
\xymatrix{
\Cor(X, Y) \otimes_\Z G(Y) 
\ar[d] \ar[r]^-{(0)}
& G(X)
\\
(\ol{h}_0(Y) \otimes G(Y))(X) \ar[d] \ar[ur]_-{(1)} \ar[rd] &
\\
(\ol{h}_0(Y)_\Nis \otimes G(Y))(X) \ar[d]
&
(\ol{h}_0(Y) \otimes G(Y))_\Nis(X) 
\ar[uu]_-{(2)} \ar[dl]^-{(3)}_-{\cong}
\\
(\ol{h}_0(Y)_\Nis \otimes G(Y))_\Nis(X). 
}
\]
The map (0) factors through (1) since $G$ is $\P^1$-invariant;
it also factors through (2) since it is a Nisnevich sheaf.
By \eqref{eq:iso-nis}, (3) is an isomorphism.
On the other hand, we have
\begin{align*}
(\ol{h}_0(Y)_\Nis \otimes G(Y))(X)
&\cong
(h_0(Y)_\Nis \otimes G(Y))(X)
\\
=
&h_0(Y)_\Nis(X) \otimes_\Z G(Y)
\cong
\CH_0(Y_{k(X)}) \otimes_\Z G(Y),
\end{align*}
where the first isomorphism is from \cite[Theorem 3.5]{KOY}
and the third from \cite[Theorem 3.1.2]{KS2}.
We obtain an induced map 
$\CH_0(Y_{k(X)}) \otimes_\Z G(Y) \to G(X)$.
The proposition follows by adjunction.
\end{proof}


\bibliographystyle{plain}

\end{document}